%% file: LowValency.tex
\documentclass[11pt]{amsart}

\usepackage[utf8]{inputenc}

\usepackage{amssymb}
\usepackage{upref}
\RequirePackage{doi}
\usepackage{hyperref}
\usepackage[capitalize]{cleveref}
\usepackage{color} 
\usepackage{mathtools}

\numberwithin{equation}{section}

\newenvironment{thmref}[1]{\begingroup\def\theequation{\ref{#1}}}{\endgroup}

\newtheorem{cor}[equation]{Corollary}
\newtheorem{lem}[equation]{Lemma}
\newtheorem{prop}[equation]{Proposition}
\newtheorem{thm}[equation]{Theorem}

\crefformat{lem}{Lemma~#2#1#3}
\Crefname{lem}{Lemma}{Lemmas}
\crefname{lem}{lemma}{lemmas}
\crefmultiformat{lem}{Lemmas~#2#1#3}{ and~#2#1#3}{, #2#1#3}{ and~#2#1#3}
\crefrangeformat{lem}{Lemmas~#3#1#4 to~#5#2#6}
\crefformat{cor}{Corollary~#2#1#3}
\Crefname{cor}{Corollary}{Corollaries}
\crefname{cor}{corollary}{corollaries}
\crefformat{prop}{Proposition~#2#1#3}
\Crefname{prop}{Proposition}{Propositions}
\crefname{prop}{proposition}{propositions}
\crefformat{thm}{Theorem~#2#1#3}
\Crefname{thm}{Theorem}{Theorems}
\crefname{thm}{theorem}{theorems}

\theoremstyle{definition}
\newtheorem{assump}[equation]{Assumption}

\newtheorem{defn}[equation]{Definition}

\newtheorem{eg}[equation]{Example}
\newtheorem{notation}[equation]{Notation}
\newtheorem*{notation*}{Notation}
\newtheorem{rem}[equation]{Remark}

\theoremstyle{remark}
\newtheorem*{ack}{Acknowledgments}

\crefmultiformat{rem}{Remarks~#2#1#3}{ and~#2#1#3}{, #2#1#3}{ and~#2#1#3}
\crefformat{rem}{Remark~#2#1#3}
\crefformat{rems}{Remark~#2#1#3}
\Crefname{rem}{Remark}{Remarks}
\crefname{rem}{remark}{remarks}

\crefformat{figure}{Figure~#2#1#3}

\newcounter{case}
\newcounter{caseholder} 
\numberwithin{case}{caseholder}
\newenvironment{case}[1][\unskip]{\refstepcounter{case}\bf
\medskip \noindent Case \thecase\ #1.\ \it}{\unskip\upshape}
\renewcommand{\thecase}{\arabic{case}}
\crefformat{case}{Case~#2#1#3}
\Crefname{case}{Case}{Cases}
\crefname{case}{case}{cases}

\newcounter{subcase}
\newenvironment{subcase}[1][\unskip]{\refstepcounter{subcase}\bf
\medskip \noindent Subcase \thesubcase\ #1.\ \it}{\unskip\upshape}
\numberwithin{subcase}{case}
\crefformat{subcase}{Subcase~#2#1#3}
\crefname{subcase}{subcase}{subcases}
\Crefname{subcase}{Subcase}{Subcases}

\newcounter{subsubcase}
\newenvironment{subsubcase}[1][\unskip]{\refstepcounter{subsubcase}\bf
\medskip \noindent Subsubcase \thesubsubcase\ #1.\ \it}{\unskip\upshape}
\numberwithin{subsubcase}{subcase}
\crefformat{subsubcase}{Subsubcase~#2#1#3}
\crefname{subsubcase}{subsubcase}{subsubcases}
\Crefname{subsubcase}{Subsubcase}{Subsubcases}

\usepackage{color}
\newcommand{\refnote}[1]{\marginpar{%
	\color{blue}
	\vbox to 0pt{\vss
	$\begin{pmatrix} \text{see} \\[-3pt] \text{note} \\[-3pt] \text{\ref{#1}} \end{pmatrix}$%
	\vskip -1.1\baselineskip}}}
	
\theoremstyle{definition}
\newtheorem{aid}{}
\numberwithin{aid}{section}
\newcommand{\oldaid}{}
\let\oldaid=\aid
\renewcommand{\aid}{\bigbreak\oldaid}
\newcommand{\oldendaid}{}
\let\oldendaid=\endaid
\renewcommand{\endaid}{\oldendaid\bigskip\hrule width\textwidth \bigbreak}
\crefformat{aid}{Note~#2#1#3}

\renewcommand{\pmod}[1]{\ (\mathrm{mod}~#1)}

\let\oldenumerate=\enumerate
\def\enumerate{\oldenumerate \itemsep=\smallskipamount}
\let\olditemize=\itemize
\def\itemize{\olditemize \itemsep=\smallskipamount}

\newcommand{\cartprod}{\mathbin{\raise0.7pt\hbox{\smaller[2]$\square$}}} 

\newcommand{\nh}{\mathord{\mathchoice
	{\hbox to 0pt{\vrule width 6pt height 2.75pt depth -1.95pt\hss}n}%
	{\hbox to 0pt{\vrule width 6pt height 2.75pt depth -1.95pt\hss}n}%
	{\hbox to 0pt{\vrule width 5pt height 2pt depth -1.6pt\hss}n}%
	{\hbox to 0pt{\vrule width 4.3pt height 1.6pt depth -1.1pt\hss}n}}}

\newcommand{\ZZ}{\mathbb{Z}}

\DeclareMathOperator{\Aut}{Aut}
\DeclareMathOperator{\Cay}{Cay}
\DeclareMathOperator{\lcm}{lcm}

\newcommand{\pref}[1]{\textup(\ref{#1}\textup)}
\newcommand{\fullref}[2]{\ref{#1}\pref{#1-#2}}
\newcommand{\fullcref}[2]{\cref{#1}\pref{#1-#2}}
\newcommand{\fullCref}[2]{\Cref{#1}\pref{#1-#2}}

\makeatletter
\newcommand{\noprelistbreak}{\smallskip\@nobreaktrue\nopagebreak} 
\makeatother

\begin{document}

\title[Unstable circulants of valency at most 7]{Automorphisms of the double cover of \\ a circulant graph of valency at most 7}

\author{Ademir Hujdurović}
\address{University of Primorska, UP IAM, Muzejski trg 2, 6000 Koper, Slovenia and University of Primorska, UP FAMNIT, Glagolja\v ska 8, 6000 Koper, Slovenia}
\email{ademir.hujdurovic@upr.si}

\author{Đorđe Mitrović}
\address{University of Primorska, UP FAMNIT, Glagolja\v ska 8, 6000 Koper, Slovenia}
\email{mitrovic98djordje@gmail.com}

\author{Dave Witte Morris}
\address{Department of Mathematics and Computer Science, University of Lethbridge, Lethbridge, Alberta, T1K~3M4, Canada}
\email{dave.morris@uleth.ca}

\date{\today}

\begin{abstract}
A graph~$X$ is said to be \emph{unstable} if the direct product $X \times K_2$ (also called the canonical double cover of~$X$) has automorphisms that do not come from automorphisms of its factors~$X$ and~$K_2$. It is \emph{nontrivially unstable} if it is unstable, connected, and non-bipartite, and no two distinct vertices of X have exactly the same neighbors.

We find all of the nontrivially unstable circulant graphs of valency at most~$7$. (They come in several infinite families.) We also show that the instability of each of these graphs is explained by theorems of Steve Wilson. This is best possible, because there is a nontrivially unstable circulant graph of valency~$8$ that does not satisfy the hypotheses of any of Wilson's four instability theorems for circulant graphs.
\end{abstract}

\maketitle

\section{Introduction}

Let $X$ be a circulant graph.
(All graphs in this paper are finite, simple, and undirected.)

\begin{defn}[\cite{CanCover}]
The \emph{canonical bipartite double cover} of~$X$ is the bipartite graph~$BX$ with $V(BX) = V(X) \times \{0,1\}$, where
	\[ \text{$(v,0)$ is adjacent to $(w,1)$ in $BX$}
	\quad \iff \quad
	\text{$v$ is adjacent to~$w$ in~$X$} . \]
\end{defn}

Letting $S_2$ be the symmetric group on the $2$-element set~$\{0,1\}$, it is clear that the direct product $\Aut X \times S_2$ is a subgroup of $\Aut BX$. We are interested in cases where this subgroup is proper:

\begin{defn}[{\cite[p.~160]{MarusicScapellatoSalvi}}]
 If $\Aut BX \neq \Aut X \times S_2$, then $X$ is \emph{unstable}.
\end{defn}

It is easy to see (and well known) that if $X$ is disconnected, or is bipartite, or has ``twin'' vertices (see \cref{TwinFreeDefn} below), 
then $X$ is unstable (unless $X$ is the trivial graph with only one vertex). The following definition rules out these trivial examples:

\begin{defn}[cf.\ {\cite[p.~360]{Wilson}}]
If $X$ is connected, nonbipartite, twin-free, and unstable, then $X$ is \emph{nontrivially unstable}. 
\end{defn}

S.\,Wilson found the following interesting conditions that force a circulant graph to be unstable. (See \cref{CayleyDefn} for the definition of the ``Cayley graph'' notation $\Cay(G, S)$.)

\begin{thm}[Wilson {\cite[Appendix~A.1]{Wilson} (and \cite[p.~156]{QinXiaZhou})}] \label{Wilson}
Let $X = \Cay(\ZZ_n, S)$ be a circulant graph, such that $n$ is even. Let $S_e = S \cap 2 \ZZ_n$ and $S_o = S \setminus S_e$. If any of the following conditions is true, then $X$ is unstable.
\noprelistbreak
	\begin{enumerate}
	\renewcommand{\theenumi}{C.\arabic{enumi}}
	\item \label{Wilson-C1}
	There is a nonzero element~$h$ of\/ $2 \ZZ_n$, such that $h + S_e = S_e$.
	
	\begingroup \renewcommand{\theenumi}{C.\arabic{enumi}$'$}
	\item \label{Wilson-C2}
	$n$~is divisible by~$4$, and there exists $h \in 1 + 2\ZZ_n$, such that
		\makeatletter \renewcommand{\p@enumii}{} \makeatother 
		\begin{enumerate}
		\item \label{Wilson-C2-So}
		$2h + S_o = S_o$, 
		and 
		\item \label{Wilson-C2-s+b}	
		for each $s \in S$, such that $s \equiv 0$ or~$-h \pmod{4}$, we have $s + h \in S$.
		\end{enumerate}
	\endgroup
	
	\begingroup \renewcommand{\theenumi}{C.\arabic{enumi}$'$}
	\item \label{Wilson-C3}
	There is a subgroup~$H$ of\/~$\ZZ_n$, such that the set
		\[ R = \{\, s \in S \mid s + H \not\subseteq S \,\} ,\]
	is nonempty and has the property that if we let $d = \gcd \bigl( R \cup \{n\} \bigr)$, then $n/d$ is even, $r/d$ is odd for every $r \in R$, and either $H \nsubseteq d \ZZ_n$ or $H \subseteq 2d \ZZ_n$.
	\endgroup
	
	\item \label{Wilson-C4}
	There exists $m \in \ZZ_n^\times$, such that $(n/2) + mS = S$.
	\end{enumerate}
\end{thm}

\begin{rem} \label{WilsonCorrections}
As explained in \cite[Rem.~3.14]{HujdurovicMitrovicMorris}, 
the two statements~\pref{Wilson-C2} and~\pref{Wilson-C3} are slightly corrected versions of the original statements of Theorems~C.2 and~C.3 that appear in~\cite{Wilson}. The correction~\pref{Wilson-C2} is due to Qin-Xia-Zhou \cite[p.~156]{QinXiaZhou}.
\end{rem}

\begin{defn}
We say that $X$ has \emph{Wilson type} \pref{Wilson-C1}, \pref{Wilson-C2}, \pref{Wilson-C3}, or~\pref{Wilson-C4}, respectively, if it satisfies the corresponding condition of \cref{Wilson}. 
\end{defn}

\begin{rem}
Wilson type of a graph needs not be unique i.e. a graph may satisfy more than one condition from \cref{Wilson}. For example, for every odd integer $k$ with $\gcd(k,3)=1$, the graph
\[\Cay(\ZZ_{8k},\{\pm 2k,\pm 3k\})\]
has Wilson type \pref{Wilson-C1} (with $h=4k$) as well as Wilson types \pref{Wilson-C3} (with $H=\{0,4k\}, R=\{\pm 3k\}$ and $d=k$) and \pref{Wilson-C4} (with $g=3$).

Additionally, the graph $\Cay(\ZZ_8,\{\pm 1,\pm 2,\pm 3\})$ has Wilson type \pref{Wilson-C1} (with $h=4$), Wilson type \pref{Wilson-C2} (with $h=1$) and Wilson type \pref{Wilson-C4} (with $g=1$).

\end{rem}

In this terminology (modulo the corrections mentioned in \cref{WilsonCorrections}), Wilson \cite[p.~377]{Wilson} conjectured that every nontrivially unstable circulant graph has a Wilson type. Unfortunately, this is not true: other conditions that force a circulant graph to be unstable are described in \cite[\S 3]{HujdurovicMitrovicMorris} 
(and these produce infinitely many counterexamples). Prior to the work in \cite{HujdurovicMitrovicMorris}, the following counterexample (which is the smallest) had been published:

\begin{eg}[Qin-Xia-Zhou {\cite[p.~156]{QinXiaZhou}}]
The circulant graph
	\[ \Cay \bigl( \ZZ_{24}, \{ \pm2, \pm3, \pm8, \pm9, \pm10\} \bigr) \]
is nontrivially unstable, but does not have a Wilson type.
\end{eg}

The main result of this paper establishes that Wilson's conjecture is true for graphs of valency at most~$7$:

\begin{thm} \label{lessthan8}
Every nontrivially unstable circulant graph of valency at most~$7$ has Wilson type \pref{Wilson-C1}, \pref{Wilson-C2}, \pref{Wilson-C3}, or~\pref{Wilson-C4}.
\end{thm}

We actually prove more precise (but more complicated) results, which show that all of the graphs in \cref{lessthan8} belong to certain explicit families (and it is easy to see that the graphs in each family have a specific Wilson type). For example:

\begin{thmref}{val5}
\begin{thm}
A circulant graph\/ $\Cay(\ZZ_n, S)$ of valency~$5$ is unstable if and only if either it is trivially unstable, or it is one of the following:
\noprelistbreak
	\begin{enumerate}
	\item
	$\Cay(\ZZ_{12k},\{ \pm s, \pm 2k, 6k \})$ with $s$~odd, 
	which has Wilson type~\pref{Wilson-C1}.
	\item
	$\Cay(\ZZ_8,\{ \pm 1, \pm 3, 4 \})$, 
	which has Wilson type~\pref{Wilson-C3},
	\end{enumerate}
\end{thm}
\end{thmref}

The following example shows that the constant~$7$ in \cref{lessthan8} cannot be increased:

\begin{eg}[{\cite[Eg.~3.10]{HujdurovicMitrovicMorris}}] 
Let $n \coloneqq 3 \cdot 2^\ell$, where $\ell \ge 4$ is even, and let
	\[ S \coloneqq \left\{ \pm 3, \pm 6, \pm \frac{n}{12}, \frac{n}{2} \pm 3 \right\} . \]
Then the circulant graph $X \coloneqq \Cay(\ZZ_n, S)$ has valency~$8$ and is nontrivially unstable, but does not have a Wilson type.
\end{eg}

Here is an outline of the paper.
After this introduction come two sections of preliminaries:
\cref{PrelimSect} presents material from the theory of normal Cayley graphs and some other miscellaneous information that will be used;
\cref{StableSect} lists some conditions that imply $X$ is not unstable.
The remaining sections each find the nontrivially unstable circulant graphs of a particular valency (or valencies). 
Namely, \cref{val4Sect} considers valencies~$\le 4$, 
whereas \cref{val5Sect,val6Sect,val7Sect} each consider a single valency ($5$, $6$, or~$7$, respectively).
The main results are 
	\cref{val3} (valency~$\le 3$), 
	\cref{val4} (valency~$4$), 
	\cref{val5} (valency~$5$), 
	\cref{val6,explicit6} (valency~$6$), 
	and 
	\cref{val7} (valency~$7$). 

\begin{ack}
The work of Ademir Hujdurovi\'c  is supported in part by the Slovenian Research Agency (research program P1-0404 and research projects N1-0102, N1-0140, N1-0159, N1-0208, J1-1691, J1-1694, J1-1695, J1-2451 and J1-9110).
\end{ack}

\section{Preliminaries} \label{PrelimSect}

For ease of reference, we repeat a basic assumption from the first paragraph of the introduction:

\begin{assump} \label{GraphsAreSimple}
All graphs in this paper are finite, undirected, and simple (no loops or multiple edges).
\end{assump}

\subsection{Basic definitions and notation}

For simplicity (and because it is the only case we need), the following definition is restricted to abelian groups, even though the notions easily generalize to nonabelian groups.

\begin{defn} \label{CayleyDefn}
Let $S$ be a subset of an abelian group~$G$, such that $-s \in S$ for all $s \in S$. 
	\begin{enumerate}
	\item The \emph{Cayley graph} $\Cay(G, S)$ is the graph whose vertices are the elements of~$G$, and with an edge from $v$ to~$w$ if and only if $w = v + s$ for some $s \in S$ (cf.\ \cite[\S1]{Li-IsoSurvey}).
	\item \label{CayleyDefn-S1}
	For $g \in G$, we let $\tilde g = (g,1)$. 
	\item \label{CayleyDefn-BX}
	Note that if $X = \Cay(G, S)$, and we let $\widetilde S = \{\, \tilde s \mid s \in S \,\}$, then
		\[ B X = \Cay \bigl( G \times \ZZ_2, \widetilde S \bigr) . \]
	\item 
	For $g \in G$, we say that an edge $\{u,v\}$ of the complete graph on $G\times \ZZ_2$ is a \emph{$g$-edge} if $v = u \pm \tilde g$. Note that $\{u,v\}$ is an edge of $BX$ if and only if it is an $s$-edge for some $s\in S$.
	\end{enumerate}
\end{defn}

\begin{notation} \label{nh}
For convenience, proofs will sometimes use the following abbreviation:
	\[ \nh = n/2 .\]
\end{notation}

We will employ the following fairly standard notation:

\begin{notation} \label{BasicNotation}
\leavevmode\noprelistbreak
	\begin{enumerate}
	\item $K_n$ is the complete graph on~$n$ vertices.
	\item $\overline{K_n}$ is the \emph{complement} of~$K_n$. It is a graph with $n$~vertices and no edges.
	\item $K_{n,n}$ is the complete bipartite graph with $n$~vertices in each bipartition set.
	\item $C_n$ is the cycle of length~$n$.
	\item For $a \in \ZZ_n$, we use $|a|$ to denote the \emph{order} of~$a$ as an element of the cyclic group~$\ZZ_n$. So
		\[ |a| = \frac{n}{\gcd(n, a)} . \]
	It does \emph{not} denote the absolute value of~$a$.
	\item $\phi$ denotes the Euler's totient function.
	\end{enumerate}
\end{notation}

Throughout the paper various notions of graph products will be used. We now recall their definitions and notation.

\begin{defn}[{\cite[pp.~35, 36, and~43]{ProductHandbook}}] \label{ProductsDefn}
Let $X$ and~$Y$ be graphs.
	\begin{enumerate}
	\item The \emph{direct product} $X \times Y$ is the graph with $V(X \times Y) = V(X) \times V(Y)$, such that $(x_1, y_1)$ is adjacent to $(x_2, y_2)$ if and only if
		\[ \text{$(x_1,x_2) \in E(X)$ \ and \ $(y_1,y_2) \in E(Y)$.} \]
	\item The \emph{Cartesian product} $X \cartprod Y$ is the graph with $V(X \times Y) = V(X) \times V(Y)$, such that $(x_1, y_1)$ is adjacent to $(x_2, y_2)$ if and only if either
		\begin{itemize}
		\item $x_1 = x_2$ and $(y_1,y_2) \in E(Y)$, or
		\item $y_1 = y_2$ and $(x_1,x_2) \in E(X)$.
		\end{itemize}
	\item \label{ProductsDefn-wreath} The \emph{wreath product} $X \wr Y$ is the graph that is obtained by replacing each vertex of~$X$ with a copy of~$Y$. (Vertices in two different copies of~$Y$ are adjacent in $X \wr Y$ if and only if the corresponding vertices of~$X$ are adjacent in~$X$.)
	 This is called the \emph{lexicographic product} in~\cite[p.~43]{ProductHandbook} (and denoted $X \circ Y$).
	\end{enumerate}
\end{defn}

\begin{defn}[Kotlov-Lovász {\cite{KotlovLovasz-twinfree}}] \label{TwinFreeDefn} 
A graph~$X$ is \emph{twin-free} if there do not exist two distinct vertices $v$ and~$w$, such that $N_X(v) = N_X(w)$, where $N_X(v)$ denotes the set of neighbors of~$v$ in~$X$.
\end{defn}

The notion of a ``block'' (or ``block of imprimitivity'') is a fundamental concept in the theory of permutation groups, but we need only the following special case:

\begin{defn}[cf.\ {\cite[pp.~12--13]{DixonMortimer}}] \label{BlockDefn}
Let $G$ be a finite abelian group. Let $X = \Cay(G, S)$ be a Cayley graph. A nonempty subset~$\mathcal{B}$ of $V(X)$ is a \emph{block} for the action of $\Aut X$ if, for every $\alpha \in \Aut X$, we have 
	\[ \text{either \ $\alpha(\mathcal{B}) = \mathcal{B}$ \ or \ $\alpha(\mathcal{B}) \cap \mathcal{B} = \emptyset$.} \]
It is easy to see that this implies $\mathcal{B}$ is a coset of some subgroup~$H$ of $G$. Then every coset of~$H$ is a block. \refnote{BlockDefn-H} Indeed, the action of $\Aut X$ permutes these cosets, so there is a natural action of $\Aut X$ on the set of cosets.
\end{defn}

\begin{rem}
The most important instance of \cref{BlockDefn} for us will be the case of canonical bipartite double covers. Indeed, if $X=\Cay(\ZZ_n,S)$ is a circulant graph then its canonical bipartite double cover $BX=\Cay(\ZZ_n\times \ZZ_2,\widetilde S)$, defined in \fullcref{CayleyDefn}{BX}, is a Cayley graph. Therefore, every block for the action of $\Aut BX$ is a coset of some subgroup $H$ of $\ZZ_n\times\ZZ_2$.
\end{rem}

\subsection{Normal Cayley graphs}

\begin{defn}
\label{CayleyDefn-Normal}(M.-Y.\,Xu \cite[Defn.~1.4]{Xu-AutAndIso})
For each $g \in G$, it is easy to see that the translation~$g^*$, defined by $g^*(x) = g + x$, is an automorphism of $\Cay(G, S)$. The set
	\[ G^* = \{\, g^* \mid g \in G \,\} \]
is a subgroup of $\Aut \Cay(G,S)$. (It is often called the \emph{regular representation} of~$G$.) We say that $\Cay(G, S)$ is \emph{normal} if the subgroup~$G^*$ is normal in $\Aut \Cay(G,S)$. This means that if $\varphi$ is automorphism of the graph $\Cay(G, S)$, and $\varphi(0) = 0$, then $\varphi$ is an automorphism of the group~$G$.
\end{defn}

\begin{lem} \label{NormalCover}
Assume $X = \Cay(\ZZ_n, S)$ is connected and unstable. If the Cayley graph $BX$ is normal, then $X$ has Wilson type~\pref{Wilson-C4}.
\end{lem}

\begin{proof}
Since $X$ is unstable, we know that $(0,1)$ is not central in $\Aut BX$. So it is conjugate to some other element of order~$2$. However, since $BX$ is normal, we know that $\ZZ_n \times \ZZ_2$ is a normal subgroup of $\Aut BX$. Therefore, $(0,1)$ cannot be the only element of order~$2$ in $\ZZ_n \times \ZZ_2$; so $n$ is even.  (This is also immediate from \cref{OddCirculant}.)

Also note that $\ZZ_n \times \{0\}$ is normal in $\Aut BX$, because it consists of the elements of the normal subgroup $\ZZ_n \times \ZZ_2$ that preserve each bipartition set. Since $(\nh,0)$ is the unique element of order~$2$ in this normal subgroup, it must be central in $\Aut BX$.

Now, since $(0,1)$, $(\nh,0)$, and~$(\nh,1)$ are the only elements of order~$2$ in $\ZZ_n \times \ZZ_2$, and $(\nh,0)$ is not conjugate to any other elements, we see that $(0,1)$ must be conjugate to~$(\nh,1)$, by some $\alpha \in \Aut BX$. 

Since $BX$ is normal, $\alpha$ must be a group automorphism. Hence, there is some $g \in \ZZ_n^\times$, such that $\alpha(s, 0) = (gs, 0)$ for all~$s$. Since $\alpha(0,1) = (\nh, 1)$, this implies $\alpha(s, 1) = (gs + \nh, 1)$. Since $S \times \{1\}$ is $\alpha$-invariant, this implies that $S$ is invariant under the map $s \mapsto gs + \nh$, which is precisely the condition of Wilson type~\pref{Wilson-C4}.
\end{proof}

\begin{cor} \label{OddValencyNotNormal}
If $X = \Cay(\ZZ_n, S)$ is a nontrivially unstable circulant graph of odd valency, then the Cayley graph~$BX$ is not normal.
\end{cor}

\begin{proof}
If $BX$ is normal, then the \cref{NormalCover} implies $\nh + gS = S$, for some $g \in \ZZ_n^\times$. Also, since $X$ has odd valency, we know that $\nh \in S$. Also, we know that $g$ is odd (because $n$ is even). Therefore
	\[ 0 = \nh + \nh =  \nh + g\nh \in \nh + gS = S .\]
This contradicts our standing assumption that all graphs are simple (no loops).
\end{proof}

\begin{lem} \label{ReduceNonNormal}
Let $X = \Cay(\ZZ_n,S)$ be a nontrivially unstable circulant graph of odd valency, and let  $X_0 = \Cay(\ZZ_n,S\setminus\{n/2\})$, so
	\[ BX_0 = \Cay(\ZZ_n\times \ZZ_2,(S\setminus\{n/2\})\times \{1\}) . \]
If every automorphism of $BX$ maps $n/2$-edges to $n/2$-edges, then $BX_0$ is not a normal Cayley graph. Moreover, if $X_0$ is bipartite, then $X_0$ is not a normal Cayley graph.
\end{lem}

\begin{proof}
By the assumption on $\nh$-edges, it follows that every automorphism of $BX$ induces an automorphism of $BX_0$. If $BX_0$ is normal, it follows that $BX$ is also normal, contradiction with Corollary~\ref{OddValencyNotNormal}. We conclude that $BX_0$ is non-normal.

Suppose now that $X_0$ is bipartite. It is not difficult to see that $X_0$ is connected, since $X$ is connected. It follows that every element of $S\setminus\{\nh\}$ is odd. Since $X$ is nonbipartite, it follows that $\nh$ is even. 
Suppose that $X_0$ is normal Cayley graph. Observe that $BX_0$ is isomorphic to the disjoint union of two copies of $X_0$, and that the connected component containing the vertex $(0,0)$ is $X_1=\Cay(H,(S\setminus\{\nh\})\times \{1\})$, where $H=\langle (1,1)\rangle\leq \ZZ_n \times \ZZ_2$. The map $\theta:\ZZ_n \to H$ defined by $\theta(k)=(k,k~\mod~2)$ is an isomorphism between $X_0$ and $X_1$. Then $X_1$ is a normal Cayley graph on $H$. Let $\varphi$ be an automorphism of $BX$ that fixes $(0,0)$. Then $\varphi$ is also an automorphism of $BX_0$, and consequently of its connected component $X_1$. Since $X_1$ is normal, it follows that the action of $\varphi$ on $H$ is an automorphism of the group $H$, hence $\varphi$ fixes the unique element of order 2 in $H$, which is $(\nh,0)$. Observe that the $\nh$-edge in $BX$ incident with $(\nh,0)$, must also be fixed, hence $(0,1)$ is fixed. This shows that $X$ is stable, a contradiction. The obtained contradiction shows that $X_0$ is non-normal.
\end{proof}

\begin{prop}[Baik-Feng-Sim-Xu {\cite[Thm.~1.1]{BaikFengSimXu4}}] \label{4cyclenormal}
Let $\Cay(G,S)$ be a connected Cayley graph on an abelian group~$G$. Assume, for all $s,t,u,v \in S$:
	\[ s + t = u + v \neq 0 \implies \{s,t\} = \{u,v\} .\]
Then the Cayley graph $\Cay(G,S)$ is normal.
\end{prop}

\begin{proof}
For the reader's convenience, we essentially reproduce the proof that is given in~\cite{BaikFengSimXu4}, but add a few additional details.

It suffices to show that for every $\sigma\in\Aut(\Cay(G,S))$, such that $\sigma(0) = 0$:
	\begin{align} \label{4cyclenormalpf-suffices}
	(s_1+\cdots+s_n)^\sigma = s_1^\sigma+\cdots+s_n^\sigma, \quad 
	\text{for $s_1,\ldots,s_n\in S$, $n\in\mathbb{N}$} 
	. \end{align}
To accomplish this, we will show that 
	\begin{align} \label{4cyclenormalpf-willshow}
	(x+u)^\sigma=x^\sigma+u^\sigma\implies (x+u+v)^\sigma=x^\sigma + u^\sigma + v^\sigma 
	\end{align}
for all $x\in G$, and $u,v\in S$ and then \pref{4cyclenormalpf-suffices} follows by induction on~$n$. 

\refstepcounter{caseholder} 

\begin{case} \label{4cyclenormalpf-neq}
Assume $u\neq v$. 
\end{case}
We claim that  
	\[ \bigl( (x+u)+S \bigr) \cap \bigl( (x+v)+S \bigr) = \{x,x+u+v\} . \]
The $\supseteq$ containment is trivial. If $y$ is an element of the left-hand side, then 
	\[ \text{$y = (x+u)+s=(x+v)+t$ \ for some $s,t\in S$.} \]
Using the assumption from the statement of the \lcnamecref{4cyclenormal} and the fact that $u\neq v$, we obtain that $u=t$ and $v=s$. In particular, $y=x+u+v$. 
This completes the proof of the claim.

Since $\sigma$ is an automorphism fixing the identity, we also know that 
	\[ \left((x+u)+S\right)^\sigma =(x+u)^\sigma+S^\sigma = (x^\sigma+u^\sigma)+S ,\]
where in the last step we apply the hypothesis of~\pref{4cyclenormalpf-willshow}. It now follows that:
	\begin{align*}
	\{x^\sigma,(x+u+v)^\sigma\}
	&=\bigl((x+u)+S\cap (x+v)+S\bigr)^\sigma
	\\&= \bigl( (x^\sigma+u^\sigma)+S \bigr) \cap \bigl( (x^\sigma+v^\sigma)+S \bigr)
	\\&= \{x^\sigma,x^\sigma+u^\sigma+v^\sigma\}
	. \end{align*}
In particular, $(x+u+v)^\sigma=x^\sigma+u^\sigma+v^\sigma$. 

\begin{case}
Assume $u = v$.
\end{case}
Since $\sigma$ is a graph automorphism, we know it is a bijection from $(x+u)+S$ onto $(x^\sigma + u^\sigma)+S$. \Cref{4cyclenormalpf-neq} shows it is also a bijection 
	\[ \text{from $\bigl( (x+u)+S \bigr)\setminus \{x+2u\}$ onto $\bigl( (x^\sigma+u^\sigma)+S \bigr) \setminus\{x^\sigma+2u^\sigma\}$.} \]
Hence, it must also hold that $(x+2u)^\sigma=x^\sigma+2u^\sigma$. This finishes the proof.
\end{proof}

\subsection{Miscellany}

The following result is a very special case of the known results on automorphism groups of Cartesian products. (Note that $C_4$ is isomorphic to $K_2 \cartprod K_2$.)

\begin{prop}[cf.\ {\cite[Thm.~6.10, p.~69]{ProductHandbook}}] \label{Aut(C4timesX)}
Let $X$ be a connected graph. If there does not exist a graph~$Y$\!, such that $X \cong K_2 \cartprod Y$\!, then 
	\[ \text{$\Aut(K_2 \cartprod X) = S_2 \times \Aut X$
	\ and \ 
	$\Aut(C_4 \cartprod X) = \Aut C_4 \times \Aut X$.} \]
\end{prop}

\begin{lem}\label{nottwinfree}
Let $X=\Cay(\ZZ_n,S)$ be a connected circulant graph of order~$n$ such that $X$ is not twin-free, and let $d$ be the valency of~$X$.
	\begin{enumerate}
	\item \label{nottwinfree-Y}
	There is a connected circulant graph~$Y$ and some $m \ge 2$, such that $X \cong Y \wr \overline{K_m}$ and $d = \delta m$, where $\delta$ is the valency of~$Y$. 
	\item \label{nottwinfree-prime}
	If $d$ is prime, then $X \cong K_{d,d}$.
	\item \label{nottwinfree-d=4}
	If $d=4$, then $X$ is isomorphic either to $K_{4,4}$ or to $C_\ell \wr \overline{K_2}$ with $\ell = |V(X)|/2$. Moreover, the unique twin of\/~$0$ in the second case is $n/2$.
	\end{enumerate}
\end{lem}

\begin{proof}
\pref{nottwinfree-Y} Let $\sim$ be the relation of being twins on $V(X)$ defined in \cref{TwinFreeDefn}, i.e., write $x\sim y$ if and only if $N_X(x)=N_X(y)$ for $x,y\in V(X)$. Note that since $X$ is assumed to have no loops, equivalence classes of $\sim$ are independent sets. Furthermore, they are clearly blocks for the action of $\Aut(X)$ since $x\sim y$ if and only if $\alpha(x)\sim \alpha(y)$ for all $\alpha \in \Aut(X)$. Since $X$ is a circulant (so in particular, a Cayley graph), by \cref{BlockDefn} the blocks are cosets of some subgroup $H$. Clearly, each block is of size $m\coloneqq|H|$. From the assumption that $X$ is not twin-free, we obtain that $m\geq 2$. It is easy to see that if $x$ and $y$ are adjacent, then $x'$ is adjacent to $y'$ for all $x'\sim x$ and $y'\sim y$. It is now clear that $X \cong Y\wr \overline{K_m}$ with $Y \coloneqq \Cay(\ZZ_n/H,\widehat{S})$, where $\ZZ_n/H$ is the quotient group and $\widehat{S}\coloneqq\{s+H:s\in S\}$.

\pref{nottwinfree-prime} By~\pref{nottwinfree-Y}, we then represent $X$ as $Y \wr \overline{K_m}$, where $Y$ is $m$-regular and connected, and $m\geq 2$. As $d = \delta m$, and $d$~is prime, it follows that $m=d$ and $\delta=1$. In particular, $Y = K_2$ and $X = K_2 \wr \overline{K_d} \cong K_{d,d}$.
(Conversely, it is clear that $K_{d,d}$ is a connected circulant graph, but is not twin-free.)

\pref{nottwinfree-d=4} By~\pref{nottwinfree-Y}, we then represent $X$ as $Y \wr \overline{K_m}$, where $Y$ is $m$-regular and connected, and $m\geq 2$. As $4=\delta m$ and $m\geq 2$, it follows that $m\in \{2,4\}$. If $m=4$, then $\delta=1$ and consequently $X\cong K_2\wr\overline{K_4}\cong K_{4,4}$. If $m= 2$, then $Y$ is connected and $2$-regular, so it is isomorphic to the cycle $C_\ell$ with $\ell=|V(X)|/2$. It follows that $X\cong C_\ell \wr \overline{K_2}$. Note that in this case, two vertices are twins in $X$ if and only if they are in the same copy of $\overline{K_2}$ (see \fullcref{ProductsDefn}{wreath}). It is clear that these $2$-element sets of twins form blocks for the action of $\Aut X$ by \cref{BlockDefn}. From \cref{BlockDefn}, it also follows that they are cosets of the subgroup of $\ZZ_n$ of order $2$. As this subgroup is $\{0,n/2\}$, the conclusion follows.
\end{proof}

\begin{prop}[{\cite[Cor.~4.6]{HujdurovicMitrovicMorris}}] \label{order2orbit}
Let $\alpha$ be an automorphism of~$BX$, where $X$ is a circulant graph $\Cay(\ZZ_n , S)$, and let $s,t \in S$. If $\alpha$ maps some $s$-edge to a $t$-edge, and either $\gcd(|s|, |t| \bigr) = 1$, or $S$ contains every element that generates $\langle s \rangle$ \textup(e.g., if $|s| \in \{2,3,4,6\}$\textup), then $S$ contains every element that generates~$\langle t \rangle$.
\end{prop}

\begin{thm}[Kovács \cite{Kovacs-ArcTransCirc}, Li \cite{Li-ArcTransCirc}] \label{ArcTrans}
Let $X$ be a connected, arc-transitive, circulant graph of order~$n$. Then one of the following holds:
	\begin{enumerate}
	\item \label{ArcTrans-Kn}
	$X = K_n$,
	\item \label{ArcTrans-normal}
	$X$ is a normal circulant graph,
	\item \label{ArcTrans-wreath}
	$X = Y\wr \overline{K_d}$, where $n = md$, $d \ge 2$, and $Y$ is a connected arc-transitive circulant graph of order~$m$,
	\item \label{ArcTrans-AlmostWreath}
	$X = Y\wr \overline{K_d} - dY$, where $n = md$, $d > 3$, $\gcd(d,m) = 1$, and $Y$ is a connected
arc-transitive circulant graph of order~$m$.
	\end{enumerate}
\end{thm}

The statement of the following result in~\cite{HujdurovicMitrovicMorris} requires the graph to have even order (because the statement refers to $n/2$), but the same proof applies to graphs of odd order. Although the proofs in this paper will only apply the \lcnamecref{2S'-BX} to graphs of even order, we omit this unnecessary hypothesis.

\begin{lem}[{\cite[Cor.~4.3]{HujdurovicMitrovicMorris}}] \label{2S'-BX}
Let $X = \Cay(\ZZ_n, S)$ be a circulant graph, let $\varphi$ be an automorphism of $BX$, and let
    \[ S' = \{\, s' \in S \mid \text{$2t \neq 2s'$ for all $t \in S$, such that $t \neq s'$} \,\} . \]
Then $\varphi$ is an automorphism of\/ $\Cay(\ZZ_n \times \ZZ_2, \ 2 S' \times \{0\})$.
\end{lem}

\begin{prop}[{\cite[Cor.~5.6(4)]{HujdurovicMitrovicMorris}}] \label{InterchangeCI-valency}
Let $X = \Cay(\ZZ_n, S)$ be a nontrivially unstable, circulant graph, such that $n \equiv 2 \pmod{4}$, and such that\/ $2\ZZ_n \times \{0\}$ is a block for the action of $\Aut BX$.
If the valency of~$X_e$ is $\le 5$, then $X$ has Wilson type~\pref{Wilson-C1} or~\pref{Wilson-C4}.
\end{prop}

\begin{prop}[cf.\ {\cite[Prop.~3.4]{GloverMarusic}}] \label{CubicArcTrans}
If $X$ is a connected, cubic, arc-transitive multigraph, and the girth of~$X$ is $\le 5$, then $X$ is one of the following graphs:
	the theta graph~$\Theta_2$ \textup(which has multiple edges\textup), 
	$K_4$, 
	$K_{3,3}$, 
	the cube $Q_3 = K_2 \cartprod K_2 \cartprod K_2$, 
	the Petersen graph $GP (5, 2)$,
	 or 
	 the dodecahedron graph $GP (10, 2)$.
\end{prop}

\begin{cor} \label{CubicArcTransCirculant}
The only connected, cubic, arc-transitive circulant graphs are $K_4$ and $K_{3,3}$.
\end{cor}

\section{Some conditions that imply stability} \label{StableSect}

\begin{thm}[Fernandez-Hujdurović \cite{FernandezHujdurovic} (or \cite{Morris-OddAbelian})] \label{OddCirculant}
There are no nontrivially unstable, circulant graphs of odd order.
\end{thm}

\begin{lem}[cf.\ {\cite[Lem.~2.4]{FernandezHujdurovic}}] \label{StableIffStabilizer}
A circulant graph $X = \Cay(\ZZ_n, S)$ is stable if and only if, for every $\alpha \in \Aut BX$, such that $\alpha(0,0) = (0,0)$, we have $\alpha(0,1) = (0,1)$.
\end{lem}

The complete graph on $2$ vertices is bipartite, and therefore unstable. It is not difficult to see that all of the larger complete graphs are stable:

\begin{eg}[{\cite[Eg.~2.2]{QinXiaZhou}}] \label{KnStable}
If $n \ge 3$, then $K_n$ is stable.
\end{eg}

\begin{lem} \label{StableSubgraph}
Let $X = \Cay(\ZZ_n, S)$ be a connected, nonbipartite circulant graph, and let $S_0$ be a nonempty subset of~$\ZZ_n\setminus\{0\}$ such that $S_0 = -S_0$. If every automorphism of~$BX$ maps $S_0$-edges to~$S_0$-edges, and some \textup(or, equivalently, every\textup) connected component of $\Cay(\ZZ_n, S_0)$ is a stable graph, then $X$ is stable.
\end{lem} 

\begin{proof}
Let $\alpha$ be an automorphism of~$BX$ that fixes $(0,0)$, and let $X_0$ be the connected component of $\Cay(\ZZ_n, S_0)$ that contains~$0$. Since $\alpha$ maps $S_0$-edges to~$S_0$-edges, it restricts to an automorphism of~$BX_0$. Since $X_0$ is a stable graph, this implies that $\alpha(0,1) = (0,1)$.
\end{proof}

\begin{lem} \label{2aneqs+t}
Let $X = \Cay(G,S)$ be a Cayley graph on an abelian group, and let $k \in \ZZ^+$, such that $k$ is odd.
Suppose there exists $c \in S$, such that
	\begin{enumerate}
	\item \label{2aneqs+t-k}
	$|c| = k$,
	\item \label{2aneqs+t-2a}
	$2 c \neq s + t$, for all $s,t \in S \setminus \{c\}$,
	and
	\item \label{2aneqs+t-2b}
	for all $a \in S$ of order~$2k$, there exist $s,t \in S \setminus \{a\}$, such that $2 a = s + t$.
	\end{enumerate}
Then $X$ is stable.
\end{lem}

\begin{proof}
Let us say that a cycle in~$BX$ is \emph{exceptional} if, for every pair $v_i, v_{i+2}$ of vertices at distance~$2$ on the cycle, the unique path of length~$2$ from $v_i$ to~$v_{i+2}$ is $v_i,v_{i+1},v_{i+2}$. It is clear that every automorphism of~$BX$ must map each exceptional cycle of length~$k$ to an exceptional cycle of length~$k$.

Let $\alpha$ be an automorphism of $BX$ fixing $(0,0)$. If $c$ is any element satisfying the conditions, then $(c,1)^{2k}$ is an exceptional cycle. Furthermore, every exceptional cycle of length~$2k$ is of this form.  Since $k (c,1) = (0,1)$, for every such exceptional cycle, this implies that $\alpha$ fixes $(0,1)$. So $X$ is stable by \cref{StableIffStabilizer}.
\end{proof}

\begin{lem}\label{NormalSubgraph}
Let $X:=\Cay(\mathbb{Z}_n,S)$ be a circulant graph of even order and odd valency. Let $S_0 \subset S$ be non-empty such that $n/2 \in \langle S_0 \rangle $. Assume that the set of $S_0$-edges is invariant under the elements of $\Aut BX$ \textup(and $S_0 = -S_0$\textup). If some \textup(equivalently every\textup) connected component~$X_0'$ of\/ $\Cay(\mathbb{Z}_n,S_0)$ is not bipartite and has the property that $BX_0'$ is normal, then $X$ is stable.
\end{lem}

\begin{proof}
We can take $X_0' \coloneqq \Cay(\langle S_0\rangle, S_0)$. Because $X$ is of odd valency, we know $\nh \in S$. By our assumptions, $\nh$ is a vertex of $X_0'$. As $X_0'$ is connected and assumed to be non-bipartite, $BX_0'$ is connected. Let $\alpha \in \Aut(BX)_{(0,0)}$. Then by our assumptions, $\alpha \in \Aut(BX_0)$. Because $\alpha$ fixes $(0,0)$, it also fixes the connected component of $BX_0$ containing it, which is $BX_0'$. As $BX_0'$ is normal, the restriction of $\alpha$ onto $BX_0'$ is a group automorphism of $\langle S_0 \rangle \times \ZZ_2$. Note that as $(0,1),(\nh,0),(\nh,1)$ are the only elements of order $2$ in $\langle S_0 \rangle \times \ZZ_2$, it follows that $\alpha$ must permute them among themselves. As $\alpha$ fixes the colors of $BX_0'$, it follows that $\alpha$ fixes $(\nh,0)$ because this is the unique element of order $2$ in the set $\langle S_0 \rangle \times 0$. Because $X$ is loopless, $0\notin S$ and the only element of order $2$ in the connection set of $BX$, which is $S \times 1$, is $(\nh,1)$. Since $\alpha$ fixes $S\times 1$ set-wise, it must hold that $\alpha$ fixes $(\nh,1)$ and consequently it also fixes $(0,1)$. It follows that $X$ is stable.
\end{proof}

Proofs in later sections assume that the following circulant graphs are known to be stable.

\begin{lem} \label{small}
Each of the following circulant graphs is stable:
	\begin{enumerate}

	\item
	valency $3$:
		\begin{enumerate}
		\item \label{small-6-2-3}
		$\Cay(\ZZ_6, \{\pm2, 3\})$,
		\end{enumerate}

	\item \label{small-12}
	valency $4$:
		\begin{enumerate}
		\item \label{small-12-2-3}
		$\Cay(\ZZ_{12}, \{\pm2, \pm3\})$,
		\item \label{small-12-3-4}
		$\Cay(\ZZ_{12}, \{\pm3, \pm4\})$,
		\end{enumerate}

	\item valency $5$:
		\begin{enumerate}
		\item \label{small-10-2-4-5}
		$\Cay(\ZZ_{10}, \{\pm2, \pm4, 5\})$,
		\item \label{small-12-1-5-6}
		$\Cay(\ZZ_{12}, \{\pm1, \pm5, 6\})$,
		\item \label{small-12-3-4-6}
		$\Cay(\ZZ_{12},\{\pm3,\pm4,6\})$,
		\end{enumerate}

	\item \label{small-val6}
	valency $6$:
		\begin{enumerate}
		\item \label{small-20-4-5-8}
		$\Cay(\ZZ_{20},\{\pm 4, \pm 5, \pm8\})$,
		\item \label{small-20-2-5-6}
		$\Cay(\ZZ_{20},\{\pm 2, \pm 5, \pm 6\})$,
		\end{enumerate}

	\item valency $7$:
		\begin{enumerate}
		\item \label{small-12-2-3-4-6}
		$\Cay(\ZZ_{12},\{\pm2,\pm3,\pm4,6\})$,
		\item \label{small-14-2-4-6-7}
		$\Cay(\ZZ_{14}, \{\pm2, \pm4, \pm6, 7\})$,
		\item \label{small-20-4-5-8-10}
		$\Cay(\ZZ_{20}, \{\pm4, \pm5, \pm8, 10\})$,
		\item \label{small-24-3-8-9-12}
		$\Cay(\ZZ_{24},\{\pm 3,\pm 8,\pm 9, 12\})$,
		\item \label{small-30-5-6-12-15}
		$\Cay(\ZZ_{30}, \{\pm5, \pm6, \pm12, 15\})$,
		\item \label{small-30-3-9-10-15}
		$\Cay(\ZZ_{30},\{\pm3,\pm9,\pm10,15\})$,
		\item \label{small-30-6-10-12-15}
		$\Cay(\ZZ_{30}, \{\pm6, \pm10, \pm12, 15\})$,
		\end{enumerate}
	\end{enumerate}
\end{lem}

\begin{proof}
This is easily verified by computer (in less than a second).
For example, the MAGMA code in \cref{MagmaCode} can be executed on \url{http://magma.maths.usyd.edu.au/calc/}, or the sagemath code in \cref{SagemathCode} can be executed on \url{https://cocalc.com}\,.
(Source code of these programs can also be found in the ``ancillary files'' directory of this paper on the \href{https://www.arxiv.org}{arxiv}.)
\end{proof}

\begin{figure}
\begin{verbatim}
IsStable := function(n, S)
    V := {0 .. (n - 1)};
    E := {{v, (v + s) mod n} : v in V, s in S};
    X := Graph<V|E>;
    K2 := CompleteGraph(2);
    BX := TensorProduct(X, K2);
    AutX := AutomorphismGroup(X);
    AutBX := AutomorphismGroup(BX);
    return  #AutBX  eq  2 * #AutX;
end function;

/*Below we give two examples of how to run this code.*/
/*Similarly, stability of any other circulant graph can*/
/*be tested by inputting its order and its connection set.*/

print IsStable(6,{2,-2,3});
print IsStable(12,{2,-2,3,-3});
\end{verbatim}
\caption{MAGMA code to verify \cref{small}}
\label{MagmaCode}
\end{figure}

\begin{figure}
\begin{verbatim}
def IsStable(n, S):
    V = [0 .. (n - 1)]
    E = [[v, (v + s) % n] for v in V for s in S]
    X = Graph([V, E])
    K2 = graphs.CompleteGraph(2)
    BX = X.categorical_product(K2)
    AutX = X.automorphism_group()
    AutBX = BX.automorphism_group()
    return  AutBX.order()  ==  2 * AutX.order()
    
# Below we give two examples of how to run this code.
# Similarly, stability of any other circulant graph can
# be tested by inputting its order and its connection set.

print(IsStable(6, [2, -2, 3]))
print(IsStable(12, [2, -2, 3, -3]))
\end{verbatim}
\caption{Sagemath code to verify \cref{small}}
\label{SagemathCode}
\end{figure}

\begin{lem}[Surowski {\cite[Prop.~2.1]{Surowski}}] \label{Surowski}
Let $X$ be a twin-free, connected, vertex-transitive graph of diameter at least~$4$. 
If every edge of~$X$ is contained in a triangle, then $X$ is stable.
\end{lem}

\begin{cor}\label{CycWr}
Let $k \geq 3$ be an integer. The graph $X_k := C_k \wr K_2$ is stable if and only if $k\neq 4$.
\end{cor}
\begin{proof}
It is clear that every edge of $X_k$ lies on a triangle. As the diameter of $X_k$ is equal to the diameter of $C_k$, for $k\geq 8$ it follows from \cref{Surowski}
that $X_k$ is stable. It is easily verified by computer that $X_k$ is stable for all $k \in \{3,..,7\}, k\neq 4$. (For example, we can note that $X_k \cong \Cay(\ZZ_{2k},\{\pm 1, k\pm 1\})$ for all $k\geq 3$ and apply the code in \cref{MagmaCode} or \cref{SagemathCode}.) The desired conclusion follows.
\end{proof}

\section{Unstable circulants of valency \texorpdfstring{$\le 4$}{at most 4}} \label{val4Sect}

\begin{thm} \label{4orless}
Every nontrivially unstable circulant graph of valency $\le 4$ has Wilson type~\pref{Wilson-C4}.
\end{thm}

The union of the following two results provides a more precise formulation of the above \lcnamecref{4orless}.

\begin{prop}
\label{val3}
There are no nontrivially unstable circulant graphs of valency $\le 3$.
\end{prop}

\begin{proof}
We consider twin-free, connected, nonbipartite, circulant graphs of each valency~$\le 3$.

(valency $0$) The one-vertex trivial graph~$K_1$ is stable, because $BK_1 = \overline{K_2}$, and $|{\Aut \overline{K_2}}| = 2 = 2 \, |{\Aut K_1}|$.

(valency $1$) $K_2$ is bipartite.

(valency $2$) If $C_n$ is a nonbipartite cycle, then $n$ is odd, so $BC_n \cong C_{2n}$, so 
	\[ |{\Aut BC_n}| = |{\Aut C_{2n}}| =  2 \cdot 2n = 2 \cdot |{\Aut C_n}| , \]
so $C_n$ is stable.

(valency $3$) A connected, nonbipartite, circulant graph~$X$ of valency~$3$ is either an odd prism or a nonbipartite M\"obius ladder. In either case, the canonical double cover is the even prism $K_2 \cartprod C_n$, where $n = |V(X)|$. It is easy to check that $K_4$ is stable (see \cref{KnStable}). And the following calculation (which uses \cref{Aut(C4timesX)}) shows that $X$ is also stable when $n > 4$:
	\[ |{\Aut BX}| = |{\Aut (K_2 \cartprod C_n)}| = |S_2 \times {\Aut C_n}| = 2 \cdot 2n \le 2 |{\Aut X|} 
	. \qedhere \]
\end{proof}

\begin{thm} \label{val4}
A circulant graph\/ $\Cay(\ZZ_n, \{\pm a, \pm b\})$ of valency~$4$ is unstable if and only if either it is trivially unstable, or one of the following conditions is satisfied\/ \textup(perhaps after interchanging $a$ and~$b$\textup)\textup:
\noprelistbreak
	\begin{enumerate}
	\item \label{val4-k2=1}
	$n \equiv 2 \pmod{4}$, $\gcd(a,n) = 1$, and $b = m a + (n/2)$, for some $m \in \ZZ_n^\times$, such that $m^2 \equiv \pm1 \pmod{n}$,
	or
	\item \label{val4-gcd}
	$n$ is divisible by~$8$ and $\gcd \bigl( |a|, |b| \bigr) = 4$.
	\end{enumerate}
In both of these cases, $X$ has Wilson type~\pref{Wilson-C4}.
\end{thm}

\begin{proof}
For convenience, let $\nh = n/2$ (see \cref{nh}).

\medbreak

($\Leftarrow$) Condition~\pref{val4-k2=1} clearly implies that $X$ has Wilson type~\pref{Wilson-C4}. (The assumption that $m^2 \equiv \pm1 \pmod{n}$ implies $mb + \nh = \pm a$.) So $X$ is unstable. 

We may now assume condition~\pref{val4-gcd} holds, which means $\gcd \bigl(|a|, |b| \bigr) = 4$.  
We may also assume that $X$ is connected (for otherwise it is trivially unstable). This implies that we may assume $a$ is odd (perhaps after interchanging $a$ and~$b$), which means $n/|a|$ is odd. Since $\gcd \bigl( |a|, |b| \bigr) = 4$, it also implies that $|b| = 4 n / |a|$.  Write $|a| = 2^r \ell$, where $\ell$ is odd. 

Now, since $\ell$ and $n/|a|$ are odd, we have $\gcd(2^r, \ell) = \gcd(2^r, n/|a|) = 1$. Also note that 
	\[ 4 = \gcd \bigl( |a|, |b| \bigr) = \gcd \bigl( 2^r \ell, 4 n / |a| \bigr) = 4 \gcd \bigl( 2^{r-2} \ell, n / |a| \bigr) . \]
Therefore $2^r$, $\ell$, and $n/|a|$ are pairwise relatively prime, so we may choose $m \in \ZZ_n^\times$, such that 
	\[ \text{$m \equiv 2^{r-1} + 1 \pmod{2^r}$, 
	\ $m \equiv 1 \pmod{\ell}$, 
	\ and 
	\ $m \equiv -1 \pmod{n/|a|}$.} \]
Then:
	\begin{itemize}
	\item[($a$)] $2(m-1)$ is divisible by $2^r \ell$, but $m-1$ is not divisible by $2^r \ell$, so $2(m-1)a = 0$, but $(m-1)a \neq 0$. Therefore $(m-1)a = \nh$, so $ma + \nh = a \in S$,
	and
	\item[($b$)] $2(m+1)$ is divisible by $4n/|a|$, but $m+1$ is not divisible by $4n/|a|$ (because $m + 1 \equiv 2 \pmod{4}$). Since $|b| = 4n/|a|$, this implies that $2(m+1)b = 0$, but $(m+1) b \neq 0$. Therefore $(m + 1)b = \nh$, so $mb + \nh = -b \in S$.
	\end{itemize}
So $m S + \nh = S$, which means that $X$ has Wilson type~\pref{Wilson-C4} (and is therefore unstable).

\medbreak

($\Rightarrow$)
Assuming that $X$ is nontrivially unstable, we will show that it satisfies the conditions of~\pref{val4-k2=1} or~\pref{val4-gcd}.
Note that $n$ must be even (see \cref{OddCirculant}). Since $X$ is connected, and not bipartite, the subgroup $2\ZZ_n$ must contain exactly one of the elements of $\{a,b\}$.

Let $\alpha$ be an automorphism of~$BX$ that fixes $(0,0)$, and is not in $\Aut X \times S_2$. Since $\ZZ_n \times \{0\}$ and $\ZZ_n \times \{1\}$ are the bipartition sets of $BX$, we know that each of these sets is $\alpha$-invariant.

\refstepcounter{caseholder} 

\begin{case} \label{val4pf-aut}
Assume $\alpha$ is a group automorphism.
\end{case}
Since $\ZZ_n \times \{0\}$ is $\alpha$-invariant, this implies there is some $m \in \ZZ_n^\times$, such that $\alpha(x,0) = (mx,0)$ for all $x \in \ZZ_n$. Since $\alpha(0,1)$ is an element of order~$2$ (and $(\nh,0)$ is fixed by~$\alpha$), we must have 
	$\alpha(0,1) \in \{(0,1), (\nh,1)\}$.
If $\alpha(0,1) = (0,1)$, then $\alpha(x,i) = (mx,i)$, which contradicts the assumption that $\alpha \notin \Aut X \times S_2$.
Therefore, we have $\alpha(0,1) = (\nh,1)$. So 
	\[ \text{$\alpha(x,i) = \bigl( mx + i\nh,i \bigr)$ \ for all $(x,i) \in BX$} . \]
Since $S \times \{1\}$ is $\alpha$-invariant, this implies that $mS + \nh = S$, so $X$ is of Wilson type~\pref{Wilson-C4}.

\begin{subcase}
Assume that $\nh$ is odd. 
\end{subcase}
Since $X$ is connected, we may assume, without loss of generality, that $a \notin 2\ZZ_n$.
Then $ma + \nh \in 2 \ZZ_n$, so we must have $\alpha(a,1) \in \{(\pm b,1)\}$. Therefore $b = ma + \nh$ (perhaps after composing $\alpha$ with the group automorphism $x \mapsto -x$, which replaces $m$ with~$-m$).

Now, we have $\alpha(a,1) = (ma + \nh, 1) = (b,1)$, so $\alpha(\pm a, 1) = (\pm b, 1)$. Since $\alpha$ is a group automorphism that preserves the set $S \times \{1\}$, this implies $\alpha(\pm b, 1) = (\pm a, 1)$, so we may write $\alpha(b,1) = (\epsilon a, 1)$ with $\epsilon \in \{\pm 1\}$. Then we have 
	$m^2 (a,1) = \alpha^2(a,1) = \epsilon(a,1)$ 
and 
	$m^2(b,1) = \alpha^2(b,1) = \epsilon(b,1)$, 
so $m^2 x = \epsilon x$ for all $x \in \ZZ_n \times \ZZ_2$. This implies $m^2 \equiv \epsilon \equiv \pm 1 \pmod{n}$. So $X$ is as described in~\pref{val4-k2=1}. 

\begin{subcase}
Assume that $\nh$ is even.
\end{subcase}
Then $m a + a$ has the same parity as~$a$ (and $m b + b$ has the same parity as~$b$), so we must have $\alpha(a,1) \in \{(\pm a, 1)\}$ and $\alpha(b,1) \in \{(\pm b, 1)\}$. There is no harm in assuming $\alpha(a,1) = (a,1)$ (by replacing $m$ with $-m$ if necessary). Then, since $\alpha$ is not the identity map, we must have $\alpha(b,1) = (-b,1)$. Therefore
	\[ \text{$\bigl( ma + \nh, 1 \bigr) = \alpha(a,1) = (a,1)$, so $(m-1)a = \nh$,} \]
and
	\[ \text{$\bigl( mb + \nh, 1 \bigr) = \alpha(b,1) = (-b,1)$, so $(m+1)b = \nh$.} \]
Since $m - 1$ and $m + 1$ are even (and $\nh$ has order~$2$), this implies that $|a|$ and $|b|$ are divisible by~$4$.

This also implies that $2(m - 1)a = 0$ and $2(m + 1)b = 0$, so $|a|$ is a divisor of $2(m - 1)$ and $|b|$ is a divisor of $2(m + 1)$. Therefore $\gcd \bigl(|a|, |b| \bigr)$ is a divisor of $\gcd \bigl( 2(m -1), 2(m + 1) \bigr) \le 4$. By combining this with the conclusion of the preceding paragraph, we conclude that $\gcd \bigl(|a|, |b| \bigr) = 4$. Then, since $X$ is not bipartite, we must have $n \equiv 0 \pmod{8}$. This establishes that conclusion~\pref{val4-gcd} holds.

\begin{case} \label{val4-neq}
Assume $2s \neq 2t$, for all $s,t \in S$, such that $s \neq t$.
\end{case}
We may assume $\alpha$ is not a group automorphism, for otherwise \cref{val4pf-aut} applies. So $BX$ is not normal. 
Therefore, \cref{4cyclenormal} implies there exist $s,t,u,v\in S$ such that  $s + t = u + v \neq 0$ and $\{s,t\} \neq \{u,v\}$. From the assumption of this \lcnamecref{val4-neq},  we see that this implies $3a = \pm b$ (perhaps after interchanging $a$ with~$b$). This implies that $a$ and $b$ have the same parity, which contradicts the assumption that $X$ is connected and nonbipartite.

\begin{case}
The remaining case.
\end{case}
Since \cref{val4-neq} does not apply, we have $2s = 2t$, for some $s,t \in S$, such that $s \neq t$. 
We may assume $s = a$.

\begin{subcase}
Assume that $t = -s = -a$. 
\end{subcase}
Then $|a| = 4$. If $n$ is divisible by~$8$, then condition~\pref{val4-gcd} is satisfied. (If $|a| = 4$, and $n$ is divisible by~$8$, then $|b|$ must be divisible by~$8$, so $\gcd\bigl( |a|, |b| \bigr) = 4$.) So we may assume $n = 4k$, where $k$ is odd. Since $X$ is nonbipartite, we know that $|b|$ is not divisible by~$4$, so the fact that $|a| = 4$ implies $|\langle a \rangle \cap \langle b \rangle| \le 2$. Hence, there is an automorphism of~$\ZZ_n$ that fixes~$a$, but inverts~$b$, so 
	\[ |{\Aut X}| \ge 4n . \]
Also, since $k$ is odd and $X$ is nonbipartite, we must have $kb \neq \pm a$. Since $4kb = nb = 0$, this implies $k(b,1) \in \{(0,1), (2a, 1)\}$. Since $(2a, 1) \notin \langle (a,1) \rangle$, this implies that $\langle (a,1) \rangle \cap \langle (b, 1) \rangle = \{(0,0)\}$, so $BX \cong C_4 \cartprod C_{n/2}$. Therefore (using \cref{Aut(C4timesX)}) we have
	\begin{align*}
	|{\Aut BX}| 
	&= |{\Aut(C_4 \cartprod C_{n/2})}|
	= |{\Aut C_4 \times \Aut C_{n/2}}|
	\\&= |{\Aut C_4}| \cdot |{\Aut C_{n/2}}|
	= 8 \cdot n
	= 2 \cdot 4n
	\le 2 \cdot |{\Aut X}|
	. \end{align*}
This contradicts the assumption that $X$ is unstable.

\begin{subcase}
Assume that $t \neq -s$. 
\end{subcase}
Therefore, we may assume $s = a$ and $t = b$, so $2a = 2b$. This means that $a - b$ has order~$2$, and must therefore be equal to $\nh$, so $S + \nh = S$. This contradicts the fact that $\Cay(\ZZ_n, S)$ is twin-free.
\end{proof}

The following observation can be verified by inspecting the list \cite{BaikFengSimXu4} of connected, non-normal Cayley graphs of valency~$4$ on abelian groups, and confirming that none of them are the canonical double cover of a nontrivially unstable circulant graph. (Recall that ``normal'' is defined in \fullcref{CayleyDefn}{Normal}.) For the reader's convenience, we provide a proof that avoids reliance on the entire classification, by extracting the relevant part of the proof in~\cite{BaikFengSimXu4} (and by using \cref{val4} to reduce the number of cases).

\begin{cor} \label{Val4CoverIsNormal}
If $X$ is a nontrivially unstable circulant graph of valency~$4$, then $BX$ is normal. 
\end{cor}

\begin{proof} 
Write $X = \Cay(\ZZ_n, \{\pm a, \pm b\})$, and suppose $BX$ is not normal. (This will lead to a contradiction.) By using \cref{4cyclenormal} (and the fact that $X$ is not bipartite), as in \cref{val4-neq} of the proof of \cref{val4}, we see that $2s = 2t$, for some $s,t \in S$, such that $s \neq t$. Therefore, we may assume that either $2a = 2b$ or $|a| = 4$.
However, we cannot have $2a = 2b$, since $X$ is twin-free. (If $b = a + \nh$, then $S = S + \nh$.)

So we have $|a| = 4$. Then, since $X$ is nontrivially unstable, we see from \cref{val4} that $n$ is divisible by~$8$ and $\langle b \rangle = \ZZ_n$. (We are now in the situation of \cite[Lem.~3.4]{BaikFengSimXu4}, but will briefly sketch the proof.) Let $\alpha \in \Aut BX$, such that $\alpha(0,0) = (0,0)$. The subgraph induced by the ball of radius~$2$ around $(0,0)$ has only two automorphisms, so the restriction of~$\alpha$ to this ball is the same as the restriction of a group automorphism~$\beta$ (such that $\beta(a) \in \{\pm a\}$ and $\beta(b) \in \{\pm b\}$). It is then easy to see that $\alpha(x) = \beta(x)$ for all~$x$, so $\alpha$ is a group automorphism. This means that $BX$ is normal.
\end{proof}

The following technical result will be used in \cref{val5Sect,val7Sect}.

\begin{cor} \label{Val4Subgraph}
Let $X = \Cay(\mathbb{Z}_n,S)$ be a circulant graph of even order and odd valency, and let $S_0 \subset S$ with $|S_0| = 4$. Let $X_0 \coloneqq \Cay(\ZZ_n, S_0)$ and let $X_0'$ be a connected component of $X_0$. Assume that the set of $S_0$-edges is invariant under $\Aut(BX)$. If either
\begin{enumerate}
    \item \label{Val4Subgraph-odd}
    $|V(X_0')|$ is odd, 
    or
    \item \label{Val4Subgraph-even}
    $X_0$ is twin-free and nonbipartite,
\end{enumerate}
then $X$ is stable.
\end{cor}

\begin{proof}
Let us first assume that $|\langle S_0 \rangle| = |V(X_0')|$ is odd. Note that as $X_0'$ is $4$-valent, it is twin-free. (Otherwise, \cref{nottwinfree} tells us that $X_0' \cong Y \wr \overline{K_m}$, where $m \geq 2$ and $Y$ is $\delta$-regular. Then $|V(X_0')| = m \, |V(Y)|$, so $m$ is odd. But also $4 = \delta m$, so $m$ is even, a contradiction.) Therefore $X_0'$ is a connected, twin-free, circulant graph of odd order, so, by \cref{OddCirculant}, it must be stable. It follows by \cref{StableSubgraph} that $X$ is stable.

Let us now suppose that $|\langle S_0 \rangle| = |V(X_0')|$ is even. It then follows that $\nh\in \langle S_0 \rangle$. By assumption~\pref{Val4Subgraph-even}, $X_0$ must be twin-free and nonbipartite. As all of its connected components are isomorphic to $X_0'$, it follows that $X_0'$ is twin-free and nonbipartite. In particular, $X_0'$ is not trivially unstable. If it is stable, we conclude that $X$ is stable by \cref{StableSubgraph}. If it is not stable, it is nontrivially unstable so by \cref{Val4CoverIsNormal} it follows that $BX_0'$ is normal. Applying \cref{NormalSubgraph}, we conclude that $X$ is stable.
\end{proof}

\section{Unstable circulants of valency \texorpdfstring{$5$}{5}} \label{val5Sect}

\begin{thm} \label{val5}
A circulant graph\/ $\Cay(\ZZ_n, S)$ of valency~$5$ is unstable if and only if either it is trivially unstable, or it is one of the following:
\noprelistbreak
	\begin{enumerate}
	\item \label{val5-Z12t}
	$\Cay(\ZZ_{12k},\{ \pm s, \pm 2k, 6k \})$ with $s$~odd, 
	which has Wilson type~\pref{Wilson-C1},
	\item \label{val5-Z8}
	$\Cay(\ZZ_8,\{ \pm 1, \pm 3, 4 \})$, 
	which has Wilson type~\pref{Wilson-C3}.
	\end{enumerate}
\end{thm}

\begin{rem} \label{val5Nontrivial}
It is easy to see that each connection set listed in \cref{val5} contains both even elements and odd elements, so none of the graphs are bipartite. Then it follows from \fullcref{nottwinfree}{prime} that the graphs are also twin-free. Therefore, a graph in the list is nontrivially unstable if and only if it is connected. And this is easy to check: 
	\begin{itemize}
	\item the graph in~\pref{val5-Z8} is connected (since~$1$ is in the connection set);
	\item a graph in \pref{val5-Z12t} is connected if and only if $s$ is relatively prime to~$k$.
	\end{itemize}
\end{rem}

The proof of \cref{val5} will use the following \lcnamecref{val5nhInvt}s.

The first \lcnamecref{BipNonnormalVal4} can be obtained by inspecting the list \cite[Cor.~1.3]{BaikFengSimXu4} of connected, non-normal, circulant graphs of valency~$4$:
	$K_5$, $C_m \wr \overline{K_2}$ (with $m\geq 3$) and $K_2 \wr \overline{K_5}-5K_2$. 
For the reader's convenience, we reproduce the relevant parts of the proof in~\cite{BaikFengSimXu4}.

\begin{lem}[cf.\ {\cite[Cor.~1.3]{BaikFengSimXu4}}] \label{BipNonnormalVal4}
If $X_0 = \Cay(\ZZ_n, \{\pm a, \pm b\})$ is a connected, bipartite, twin-free, non-normal, circulant graph of valency~$4$, then $X_0 = \Cay(\ZZ_{10}, \{\pm1, \pm 3\})$. 
\end{lem}

\begin{proof}[Proof \normalfont(cf.\ {\cite[proof of Thm.~1.2]{BaikFengSimXu4}})]
From \cref{4cyclenormal}, we see that (perhaps after permuting the generators), we have either $2a = 2b$ or $b = 3a$ or $4a = 0$.  However, we cannot have $2a = 2b$, since $X$ is twin-free. So there are two cases to consider.

\refstepcounter{caseholder} 

\begin{case}
Assume $b = 3a$. 
\end{case}
Since $|b| \neq 2$, we have $n > 6$. 

For $n = 8$, we have $2b = 6a = -2a$, which contradicts the assumption that $X$ is twin-free. 

For $n = 10$, we have the Cayley graph that is specified in the statement of the \lcnamecref{BipNonnormalVal4}. 

For $n \ge 12$, we have $X \cong \Cay(\ZZ_n, \{\pm 1, \pm 3\})$. This is the situation of \cite[Lem.~3.5]{BaikFengSimXu4}, but, for completeness, we sketch the proof. 
The only nontrivial automorphism of the ball of radius~$2$ centered at~$0$ is $x \mapsto -x$. From this, it easily follows that $x \mapsto -x$ is the only nontrivial automorphism of~$X$, so $X$ is normal. This contradicts our hypothesis.

\begin{case}
Assume $4a = 0$.
\end{case}
This means $|a| = 4$. Since $X$ is bipartite, we know that $|b|$ is even, so $\langle b \rangle$ contains the unique element of order~$2$; this means $2a = \ell b$ for some $\ell \in \ZZ$. Since $X$ is bipartite, we know that $\ell$ is even. So $|b|$ is divisible by~$4$. Therefore, $\langle b \rangle$ contains~$a$. So $X \cong \Cay(n, \{\pm 1, \pm n/4\})$. If we consider the subgraph induced by the ball of radius $2$, we note that only the vertices $\pm 1$ have a pendant edge. In particular, every automorphism of $X$ maps $1$-edges to $1$-edges and is therefore an automorphism of the graph $X_0\coloneqq\Cay(\ZZ_n,\{\pm 1\}) \cong C_n$. Since this Cayley graph is normal, we can conclude the same about $X$, which is a contradiction.
\end{proof}

\begin{lem} \label{val5nhInvt}
Let $X = \Cay(\ZZ_n, S)$ be a nontrivially unstable, circulant graph of valency~$5$. 
If every automorphism of $BX$ maps $n/2$-edges to $n/2$-edges, then $X = \Cay(\ZZ_8,\{ \pm 1, \pm 3, 4 \})$.
\end{lem}

\begin{proof}
Recall that \cref{nh} introduced $\nh$ as an abbreviation for $n/2$. Since $X$ has odd valency, we know that $n$ is even; indeed, we may write 
	\[ X = \Cay(\ZZ_n, \{\pm a, \pm b, \nh\}) .\]
Since every automorphism of $BX$ maps $\nh$-edges to $\nh$-edges, we know that every automorphism of $BX$ is an automorphism of $BX_0$, where 
	\[ X_0 = \Cay(\ZZ_n, \{\pm a, \pm b\}) .\] 
If $X_0$ is stable, then by \cref{StableSubgraph} it follows that $X$ is stable, a contradiction. So we may assume now that $X_0$ is unstable.

\refstepcounter{caseholder} 

\begin{case}
Assume $X_0$ is nontrivially unstable.
\end{case}
As $X_0$ is $4$-valent, by applying \cref{Val4CoverIsNormal} we conclude that $BX_0$ is a normal Cayley graph. Because every automorphism of $BX$ is an automorphism of $BX_0$, it then follows that $BX$ is normal as well. However, since $X$ is nontrivially unstable and of valency $5$, \cref{OddValencyNotNormal} implies that $BX$ is not normal, a contradiction.

\begin{case}
Assume $X_0$ is trivially unstable.
\end{case}
There are three possibilities to consider:

\begin{subcase}
Assume $X_0$ is not connected.
\end{subcase}
Then $a$ and $b$ generate a proper subgroup of $\mathbb{Z}_n$, while $a,b$ and $\nh$ generate the whole group. From here, $n = 2k$, where $k$ is odd, and $\langle a,b \rangle = 2 \ZZ_n$ has order~$k$. The connected components of $X_0$ then have order~$k$, and therefore have odd order. By applying \fullcref{Val4Subgraph}{odd} we conclude that $X$ is stable, a contradiction.

\begin{subcase}
Assume $X_0$ is connected, but is not twin-free.
\end{subcase}
Then  (by \fullcref{nottwinfree}{Y}) we can represent $X_0$ as a wreath product $Y \wr \overline{K_m}$, where $Y$ is a $\delta$-regular connected graph and $m > 1$ is an integer such that $\delta m = 4$. 

\begin{subsubcase}
Assume $m = 4$.
\end{subsubcase}
Then $\delta = 1$, so we get that $X_0 = K_2 \wr \overline{K_4} = K_{4,4}$. Hence, $X_0$ is a connected, $4$-valent Cayley graph on~$\ZZ_8$ and its connection set can only contain odd numbers, because it is also bipartite. This uniquely determines $X_0$ as $\Cay(\ZZ_8,\{\pm1, \pm3\})$. From here, $X = \Cay(\ZZ_8,\{\pm1,\pm3, 4\})$, so $X$ is the graph in the statement of the \lcnamecref{val5nhInvt}. 

\begin{subsubcase}
Assume $m = 2$.
\end{subsubcase}
Then $\delta = 2$ and 
	\[ |V(Y)| = |V(X)|/m = 2k/2 = k , \]
so $Y$ is a $k$-cycle, so $X_0 \cong C_k \wr \overline{K_2}$. Consequently, $X \cong C_k \wr K_2$. From \cref{CycWr}, and the fact that $X$ is not stable, we conclude that $k=4$,  i.e., $X \cong C_4 \wr K_2$. It is easy to see that this implies $X = \Cay(\ZZ_8, \{\pm1, \pm3,4\})$. So $X$ is again the graph in the statement of the \lcnamecref{val5nhInvt}.

\begin{subcase}
Assume $X_0$ is bipartite, connected and twin-free. 
\end{subcase}
By applying \cref{ReduceNonNormal}, we conclude that the Cayley graph~$X_0$ is not normal. Then \cref{BipNonnormalVal4} tells us that $X_0 = \Cay(\ZZ_{10},\{\pm1,\pm3\})$, meaning that $X = \Cay(\ZZ_{10},\{\pm1,\pm3,5\})$. But then $X$ is bipartite, a contradiction. 
\end{proof}

The following simple observation provides a converse to \cref{val5nhInvt}.

\begin{lem} \label{val5Z8}
Let $X = \Cay(\ZZ_8,\{ \pm 1, \pm 3, 4 \})$. Then:
	\begin{enumerate}
	\item \label{val5Z8-WilsonType}
	$X$ is nontrivially unstable, and has Wilson type~\pref{Wilson-C3}, 
	\item \label{val5Z8-invt}
	every automorphism of $BX$ maps $n/2$-edges to $n/2$-edges,
	and
	\item \label{val5Z8-orbits}
	all other edges of $BX$ are in a single orbit of $\Aut BX$.
	\end{enumerate}
\end{lem}

\begin{proof}
\pref{val5Z8-WilsonType}
$X$ has Wilson type \pref{Wilson-C3} with parameters $H = \langle  2\rangle = \{0,2,4,6\}$, $R=\{4\}$, and $d=4$. (Then $n/d = 2$ is even, $r/d = 1$ for the unique element~$r$ of~$R$, and $H = 2 \ZZ_8 \not\subset 4 \ZZ_8 = d\ZZ_8$.) See \cref{val5Nontrivial} for an explanation that $X$ is therefore nontrivially unstable.

(\ref{val5Z8-invt}~and~\ref{val5Z8-orbits})
Let 
	\[ X_0 = \Cay(\ZZ_8,\{ \pm 1, \pm 3\}) \cong K_{4,4} . \]
Since $X_0$ is bipartite, we know that $BX_0$ is isomorphic to the disjoint union of two copies of~$X_0$. But $BX$ is connected, and the element $(\nh, 1)$ of order~$2$ is the only element of its connection set that is not in the connection set of~$BX_0$. It follows that
	\[ BX \ \cong \ X_0 \cartprod K_2 \ \cong \ K_{4,4} \cartprod K_2 . \]
So we see from \cref{Aut(C4timesX)} that the set of $\nh$-edges is invariant under all automorphisms of $BX$. On the other hand,  $K_{4,4}$ is edge-transitive, so the other edges are all in a single orbit.
\end{proof}

\begin{lem} \label{val5nhNotInvt}
Let $X = \Cay(\ZZ_n, S)$ be a nontrivially unstable, circulant graph of valency~$5$. 
If the set of $n/2$-edges is not invariant under $\Aut BX$, then:
	\begin{enumerate}
	\item \label{val5nhNotInvt-Z12t}
	$X = \Cay(\ZZ_{12k},\{ \pm s, \pm 2k, 6k \})$, for some $s,k \in \ZZ^+$,  with $s$~odd,
	and
	\item \label{val5nhNotInvt-orbits}
	$\Aut BX$ has exactly two orbits on the set of edges of $BX$.
	\end{enumerate}
\end{lem}

\begin{proof}
Write 
	\[ X = \Cay(\ZZ_n, \{\pm a, \pm b, \nh\}) .\]
By assumption, some automorphism of~$BX$ maps an $\nh$-edge to an $a$-edge (perhaps after interchanging $a$ and~$b$).
Then \cref{order2orbit} shows that every generator of $\langle a \rangle$ is in $S \setminus \{\nh\}$. So the number of generators of~$\langle a \rangle$ is $\le 4$ (and $|a| > 2$), so $|a| \in \{3, 4, 5, 6, 8, 12\}$.

\refstepcounter{caseholder} 

\begin{case} \label{val5nhNotInvtPf-a5812}
Assume $|a| \in \{5, 8, 12\}$.
\end{case}
The four generators of $\langle a \rangle$ are in~$S$, so they must coincide with $\pm a$ and~$\pm b$. Therefore, $\langle a, \nh \rangle = \langle a, b, \nh\rangle = \ZZ_n$, so $n \in \{10, 8, 12\}$. Therefore, $X$ is one of the following Cayley graphs: 
	\[ \text{$\Cay(\mathbb{Z}_{10},\{\pm2, \pm4, 5\})$,
	$\Cay(\mathbb{Z}_8,\{\pm1, \pm 3, 4 \})$,
	or
	$\Cay(\mathbb{Z}_{12},\{\pm1, \pm5, 6 \})$.} \] 
Note that the first and third graphs appear in \cref{small} under \pref{small-10-2-4-5} and \pref{small-12-1-5-6} respectively, and are therefore both stable. \fullCref{val5Z8}{invt} implies that the second graph is also not possible (since the statement of the \lcnamecref{val5nhNotInvt} requires that the set of $\nh$-edges is not invariant under $\Aut BX$).
So this \lcnamecref{val5nhNotInvtPf-a5812} cannot occur.

\begin{case} \label{val5nhNotInvtPf-a346}
Assume $|a| \in \{3,4,6\}$.
\end{case}
Note that then $\pm a$ are the only generators of $\langle a \rangle$. Because all elements of $S$ are pairwise distinct, it follows that $\langle a \rangle \neq \langle b \rangle$. Therefore, $|a|\neq |b|$.

\begin{subcase}
Assume $|b|\in \{3,4,6\}$.
\end{subcase}
We consider each of the three possibilities for $\{ |a|, |b| \}$:

\begin{subsubcase}
Assume $\{ |a|, |b| \} = \{3, 4\}$. 
\end{subsubcase}
Then 
	\[ X = \Cay(\mathbb{Z}_{12},\{\pm3, \pm4, 6\}) . \]
This graph is stable by \fullcref{small}{12-3-4-6}.

\begin{subsubcase}
Assume $\{ |a|, |b| \} = \{3, 6\}$. 
\end{subsubcase}
Then 
	\[ X = \Cay(\mathbb{Z}_6,\{\pm1, \pm2,3\}) \cong K_6 . \]
This graph is stable by \cref{KnStable}.

\begin{subsubcase}
Assume $\{ |a|, |b| \} = \{4, 6\}$. 
\end{subsubcase}
Then 
	\[ X = \Cay(\mathbb{Z}_{12},\{\pm2, \pm 3, 6\}) . \]
This graph appears in part~\pref{val5nhNotInvt-Z12t} of the statement of the \lcnamecref{val5nhNotInvt} with parameters $s = 3$ and $k = 1$.

Also note that 
	\[ BX = \Cay( \ZZ_{12} \times \ZZ_2 , \{ \pm(2,1), \pm(3,1), (6,1) \}) .\]
Since $\langle (3,1) \rangle \cap \langle (2,1) , (6,1) \rangle = \{0,0\}$, this implies
	\[ BX
	\cong C_4 \cartprod \Cay( 2\ZZ_{12} \times \ZZ_2 , \{ \pm(2,1), (6,1) \})
	\cong C_4 \cartprod M_6
	, \]
where $M_6$ is the M\"obius ladder with 6~vertices.
Then \cref{Aut(C4timesX)} implies that $BX$ is not edge-transitive. Since $\nh$-edges are in the same orbit as $a$-edges, this establishes part~\pref{val5nhNotInvt-orbits} of the statement of the \lcnamecref{val5nhNotInvt} for this graph.

\begin{subcase} \label{val5nhNotInvtPf-a346-bNot346}
Assume $|b|\not\in \{3,4,6\}$.
\end{subcase}
From here, we see from \cref{order2orbit} that no automorphism of $BX$ can map an $\nh$-edge to a $b$-edge (because $S$ cannot contain more than $2$ generators of $\langle b \rangle$, in addition to $\pm a$). Hence, the set of $b$-edges is invariant under all automorphisms of $BX$. (Note that this establishes part~\pref{val5nhNotInvt-orbits} of the statement of the \lcnamecref{val5nhNotInvt} for this \lcnamecref{val5nhNotInvtPf-a346-bNot346}.) Now, we see that every automorphism of $BX$ is also an automorphism of the graphs 
	\begin{align*}
	BX_1 &\coloneqq\Cay \bigl(\ZZ_n \times \ZZ_2,\{(\pm a,1),(\nh,1)\}\bigr) \\
	\intertext{and}
	BX_2 &\coloneqq\Cay \bigl(\ZZ_n \times \ZZ_2, \{(\pm b,1)\}) ,
	\end{align*}
which are the canonical double covers of 
	\[ \text{$X_1\coloneqq\Cay \bigl(\ZZ_n, \{\pm a, \nh\})$ 
	\ and \ 
	$X_2\coloneqq\Cay \bigl(\ZZ_n, \{\pm b\})$,} \]
respectively. Note that $BX_1$ is arc-transitive (because $a$-edges and $\nh$-edges are in the same orbit of $\Aut BX$).

If $|a|=3$, then every connected component of $BX_1$ is isomorphic to a $6$-prism, which is not arc-transitive. 
If $|a|=4$, then every connected component of $X_1$ is isomorphic to $K_4$, which is a stable graph by \cref{KnStable}, so it follows from \cref{StableSubgraph} that $X$ is stable, a contradiction.
Therefore, we must have $|a|=6$. 

The connected components of $X_2$ are $|b|$-cycles. If $|b|$ is odd, then these are stable (by \cref{OddCirculant}, for example). By another application of \cref{StableSubgraph}, it follows that $X$ is stable, a contradiction.

So we can now assume that $|a| = 6$ and $|b|$ is even. Write $n=6\ell$. From here $\nh = 3\ell$ and $\{\pm a\} = \{\ell,5\ell\}$. 

Note that if $\ell$ is odd, then $a$ and $b$ must both be odd (since $|a|$ and $|b|$ are even). Since $\nh = 3\ell$ is also odd, this means that all elements of~$S$ are odd, so $X$ is bipartite, a contradiction. 

Therefore, we know that $\ell$ is even, so we may write $\ell = 2k$. Then $n = 12k$, $\nh = 6k$ and $\{\pm a\} = \{\pm 2k\}$. In particular, $\pm a$ and~$\nh$ are all even. So $c$ must be odd (since $X$ is connected). This means that $X$ appears in part~\pref{val5nhNotInvt-Z12t} of the statement of the \lcnamecref{val5nhNotInvt} with parameter $s = c$.
\end{proof}

\begin{proof}[Proof of \cref{val5}] 
$(\Leftarrow)$ It suffices to show that each of the graphs in \pref{val5-Z12t} and \pref{val5-Z8} has the specified Wilson type. 
For any member of \pref{val5-Z12t}, it holds that $S_e = \{2k,6k,10k\}$; therefore $4k+S_e = S_e$, so the graph has Wilson type~\pref{Wilson-C1}.
For the graph in~\pref{val5-Z8}, see \fullcref{val5Z8}{WilsonType}.

$(\Rightarrow)$ Assume $X = \Cay(\ZZ_n, S)$ has valency~$5$, and is nontrivially unstable. Then either \cref{val5nhInvt} or \cref{val5nhNotInvt} must apply (depending on whether the set of $\nh$-edges is invariant or not). So $X$ is one of the graphs listed in these two \lcnamecref{val5nhInvt}s, and is therefore listed in the statement of the \lcnamecref{val5}.
\end{proof}

Combining \cref{val5nhInvt,val5Z8,val5nhNotInvt} also yields the following observation that will be used in \cref{val7Sect}:

\begin{cor} \label{val5orbits}
If $X$ is a nontrivially unstable, $5$-valent, circulant graph, then $\Aut BX$ has exactly two orbits on the edges of $BX$.
\end{cor}

\section{Unstable circulants of valency \texorpdfstring{$6$}{6}} \label{val6Sect}

See \cref{explicit6} for a more explicit formulation of the following \lcnamecref{val6}.

\begin{thm} \label{val6}
Every nontrivially unstable, circulant graph of valency\/~$6$ has Wilson type~\pref{Wilson-C1}, \pref{Wilson-C2}, \pref{Wilson-C3}, or~\pref{Wilson-C4}.
\end{thm}

\begin{proof}
Let $X = \Cay(\ZZ_n, S)$ be a nontrivially unstable, circulant graph of valency~$6$, and write $S = \{\pm a, \pm b, \pm c\}$. The proof is by contradiction, so assume that $X$ does not have any of the four listed Wilson types.
As usual, we let $\nh = n/2$, for convenience.
The proof considers several \lcnamecref{val6pf-2a=2b}s.

\refstepcounter{caseholder} 

\begin{case} \label{val6pf-2a=2b}
Assume $2a = 2b$.
\end{case}
This means $b = a + \nh$ (and $-b = -a + \nh$). Since $X$ does not have Wilson type~\pref{Wilson-C1}, we also know that 
	\[ (S \cap 2\ZZ_n) \neq \nh + (S \cap 2\ZZ_n) . \]
Therefore, we must have $S \cap 2\ZZ_n \neq \{\pm a, \pm b\}$. Since $S \not\subseteq 2 \ZZ_n$, this implies 
	\begin{align} \label{NotBothEven}
	\{a,b\} \not\subseteq 2\ZZ_n 
	. \end{align}

We claim that $|c|$ is not divisible by~$4$. To see this, note that otherwise $n/{\gcd(c,n)} = |c|$ is even, so $c/{\gcd(c,n)}$ is odd, and we also know that $|2c|$ is even, so $\nh \in \langle 2c \rangle$. This contradicts the fact that $X$ does not have Wilson type~\pref{Wilson-C3} (with $H = \langle \nh \rangle$, $R = \{\pm c\}$, and $d = \gcd(c, n)$). This completes the proof of the claim.

Also note that $2c \notin \{\pm 2a, \pm 2b\}$. (For example, if $2c = 2a$, then, since $c \neq a$, we must have $c = a + \nh = b$, which contradicts the fact that $a$, $b$, and~$c$ must be distinct, because $|S| = 6$.) Furthermore, since $|c|$ is not divisible by~$4$, we also know that $|c| \neq 4$, so $2c \neq -2c$. Thus, we have
        \begin{align} \label{val6pf-2c}
        2c \notin \{\pm 2a, \pm 2b, -2c \}
        . \end{align}
Therefore, we see from \cref{2S'-BX} (with $S' = \{\pm c\}$) that
        \begin{align} \label{val6pf-Aut(2S)}
        \begin{matrix}
        \text{every automorphism of $BX$ is also an} \\
        \text{automorphism of $\Cay(\ZZ_n \times \ZZ_2 , \{(\pm 2c, 0)\} \bigr)$.}
        \end{matrix}
        \end{align}

We now consider two \lcnamecref{val6pf-only}s.

\begin{subcase} \label{val6pf-only}
Assume $(c,1)$ is the only common neighbor of $(0,0)$ and $(2c,0)$ in~$BX$.
\end{subcase}
Let $\alpha$ be an automorphism of~$BX$ that fixes $(0,0)$, but does not fix $(0,1)$.
Combining~\pref{val6pf-Aut(2S)} with the assumption of this \lcnamecref{val6pf-only} implies that $\alpha$ must preserve the set of $c$-edges.

If $|c|$ is odd, then letting $S_0 = \{\pm c\}$ in \cref{StableSubgraph} implies that $X$ is stable (because cycles of odd length are stable), which is a contradiction.

Therefore, since $|c|$ is not divisible by~$4$, we must have $|c| \equiv 2 \pmod{4}$, so $(|c|/2) \cdot (c,1) = (\nh,1)$, so this implies that
        \[ \text{$\alpha \bigl( v + (\nh,1) \bigr) = \alpha(v) + (\nh,1)$ for every vertex $v$ of~$BX$.} \]
Also note that, since $\alpha$ preserves the set of $c$-edges in $BX$, it must also preserve the complement, which consists of the $a$-edges and $b$-edges. Hence, $\alpha$ is an automorphism of the canonical double cover of the $5$-valent circulant graph $X' = \Cay(\ZZ_n, S')$, where $S' = \{\pm a, \pm b, \nh\}$.

Let $X'_0$ be the connected component of~$X'$ that contains~$0$.
Note that $X'_0$ is not stable (since $\alpha$ does not fix $(0,1)$). Also note that $X'_0$ is connected, by definition. Furthermore, it is not bipartite, because $X$ is not bipartite and $\nh = kc$ where $k$ is odd. We therefore see from \cref{nottwinfree} that it is also twin-free. So $X'_0$ is nontrivially unstable, and must therefore be one of the graphs listed in \cref{val5} (after identifying the cyclic group $V(X_0')$ with some~$\ZZ_m$ by a group isomorphism). Since $2a = 2b$, it follows that $X'_0$ is the graph in part~\pref{val5-Z8} of the \lcnamecref{val5}, so
        \[ X' = \Cay(\ZZ_n,\{\pm n/8, \pm 3n/8, \nh\}) . \]
Therefore, if we write $|c| = 2k$, where $k$ is odd, then
        \[ X = \Cay(\ZZ_{8k}, \{ \pm k, \pm 3k, \pm c \} ) . \]
So $X$ has Wilson type~\pref{Wilson-C3}, with $H = \langle 2k \rangle$, $R = \{\pm c\}$,
        \[ d = \gcd(c, 8k) = \frac{8k}{|c|} = \frac{8k}{2k} = 4 , \]
and $H \nsubseteq d \ZZ_{8k}$.

\begin{subcase}\label{val6pf-notonly}
Assume $(c,1)$ is not the only common neighbor of $(0,0)$ and $(2c,0)$ in~$BX$.
\end{subcase}
This implies that
        \[ \text{$2c$ is equal to either $-2c$ or $a - c$ or $2a$ or $a + b$ or $a-b$} \]
(perhaps after interchanging $a$ with~$b$ and/or replacing both of them with their negatives).

However, we know from~\pref{val6pf-2c} that $2c \neq -2c$ and $2c \neq 2a$. Also, if $2c=a-b=\nh$, then $|c| = 4$, which contradicts the fact that $|c|$ is not divisible by~$4$. Thus, only two possibilities need to be considered.

\begin{subsubcase} \label{val6pf-notonly-a+b}
Assume $2c = a + b = 2a + \nh$.
\end{subsubcase}
Since $2a$ and $2c$ are even, this implies $\nh$ is even, which means $n \equiv 0 \pmod{4}$. It also implies that $a$ and $b$ have the same parity, so we conclude from~\pref{NotBothEven} that $a$ and~$b$ are odd. Since $X$ is not bipartite, then $c$ must be even. Therefore,
        \[ 0 \equiv 2c = 2a + \nh \equiv 2 + \nh \pmod{4}, \]
so $\nh \equiv 2 \pmod{4}$, which means $n/4$ is odd. Also note that $4c = 4a = 4b$. Therefore, setting $k \coloneqq n/4$, we get that
        \[ X = \Cay(\ZZ_{4k},\{\pm c, \pm (c + k), \pm (c - k) \}) . \]
If $c \equiv 0 \pmod{4}$, this graph has Wilson type \pref{Wilson-C2} (with $h = k$).

So we may assume $c \equiv 2 \pmod{4}$.
We will show that this implies $X$ is stable (which is a contradiction).
Any two vertices of $BX$ that are in the same coset of the subgroup $\langle (k, 0) \rangle$ have $4$ common neighbors, \refnote{val6pf-same-coset} but no two vertices of~$BX$ that are in different cosets of $\langle (k, 0) \rangle$ have more than~$3$ common neighbors. \refnote{val6pf-diff-cosets} Therefore, each coset of $\langle (k, 0) \rangle$ is a block for $\Aut BX$.
Also note that $(c + 2k,1)$ is the only element of the coset $(c,1) + \langle (k, 0) \rangle$ that is not adjacent to $(0,0)$. So every automorphism of~$BX$ must preserve the set of $(c + 2k)$-edges.
Since $c + 2k$ is an element of odd order (because $c + 2k \equiv 2 + 2 \equiv 0 \pmod 4$ and $n = 4k$ is not divisible by~$8$), we conclude from \cref{StableSubgraph} that $X$ is stable.

\begin{subsubcase}
Assume $2c = a - c$.
\end{subsubcase}
This implies $a = 3c$, so
        \[ \ZZ_n = \langle a, b, c \rangle = \langle 3c, 3c + \nh, c \rangle = \langle c, \nh \rangle . \]

If $|c|$ is even, we know that $\nh \in \langle c \rangle$, so $\langle c \rangle = \ZZ_n$.
This implies that $c$ is odd, so $\{a,b\} \cap 2\ZZ_n \neq \emptyset$ (because $X$ is not bipartite). Since $\{a,b\} \not\subseteq 2\ZZ_n$, this implies $\nh$ is odd (i.e., $n \equiv 2 \pmod{4}$). And \cref{2S'-BX} (together with~\pref{val6pf-2c}) implies that $\langle 2(c,1) \rangle = 2 \ZZ_n \times \{0\}$ is a block for the action of $\Aut BX$. Then, since $|S \cap 2\ZZ_n| = 2 \le 5$, we see from \cref{InterchangeCI-valency} that $X$ has Wilson type~\pref{Wilson-C1} or~\pref{Wilson-C4}.

We may now assume that $|c|$ is odd.
Since $2a=2b$ and we may assume that \cref{val6pf-notonly-a+b} does not apply, it is easy to see that $X$ is stable by \cref{2aneqs+t}, which is a contradiction.

\begin{case} \label{val6pf-a=4}
Assume $|a| = 4$, and the previous \lcnamecref{val6pf-2a=2b} does not apply.
\end{case}

\begin{subcase}
Assume some automorphism of $BX$ maps an $a$-edge to a $b$-edge.
\end{subcase}
Then \cref{order2orbit} tells us that $S$ contains every generator of $\langle b \rangle$. Since $|S| = 6$ (and the only generators of $\langle a \rangle$ are $\pm a$), it follows that $\phi(|b|) \leq 4$. We also know $|b|\geq 3$, so we conclude that 
	\[ |b|\in\{3,4,5,6,8,10,12\} . \]

 Also note that if $\phi(|b|) = 4$, then $\{\pm b, \pm c\}$ consists of the 4 generators of~$\langle b \rangle$, so $n = \lcm\bigl( |a|, |b| \bigr) = \lcm\bigl( 4, |b| \bigr)$.

\smallbreak

$\bullet$ If $|b|=4$, then $a$ and $b$ are generators of the same cyclic subgroup of order~$4$, but then $\{\pm a\} = \{\pm b\}$, a contradiction. 

\smallbreak

$\bullet$ If $|b| \in \{5,10\}$, then $\phi(|b|) = 4$, so $\{\pm b, \pm c\}$ consists of the 4 generators of~$\langle b \rangle$, and we have $n = \gcd\bigl( 4, |b| \bigr) = 20$. If $|b| = 5$, then $X$ is isomorphic to 
	$ \Cay(\ZZ_{20},\{ \pm 4, \pm 5, \pm 8 \}) $;
if $|b|=10$, then $X$ is isomorphic to 
	$\Cay(\ZZ_{20},\{ \pm 2, \pm 5, \pm 6 \}) $.
Both of these graphs are stable by \fullcref{small}{val6}.

\smallbreak

$\bullet$ If $|b| \in \{8,12\}$, then $\{\pm a \} \subseteq \langle b \rangle$, and $a$ is not a generator of this subgroup. So the other two generators of $\langle b \rangle$ must be $\pm c$. This implies that $b$ and $c$ each generate the whole group $\ZZ_n$. If $|b| = 8$, then $X$ is isomorphic to 
	$\Cay(\ZZ_8,\{ \pm 1, \pm 2, \pm 3 \})$, 
which is not twin-free since $4 + S = S$. If $|b| = 12$, then $X$ is isomorphic to 
	$\Cay(\ZZ_{12},\{ \pm 1, \pm 3, \pm 5 \})$,
which is bipartite. 

\smallbreak

$\bullet$ We may now assume $|b| \in \{3,6\}$. 
We consider two \lcnamecref{val6pf-a=4-atob-36}s.

\begin{subsubcase} \label{val6pf-a=4-atob-36}
Assume $|b| \in \{3,6\}$ and $|c| \in \{3,6\}$. 
\end{subsubcase}
If $|c|=|b|$, then $b$ and~$c$ generate the same cyclic subgroup, which has only two generators, which contradicts the fact that $\{\pm b\} \neq \{\pm c\}$. So we may assume $|b|=3$ and $|c|=6$. Then $X$ is isomorphic to 
	$\Cay(\ZZ_{12},\{ \pm 2, \pm 3, \pm 4 \})$,
which is not twin-free since $6 + S = S$. 

\begin{subsubcase}
Assume $|b| \in \{3,6\}$ and $|c| \notin \{3,6\}$. 
\end{subsubcase}
Then we can assume that no isomorphism of $BX$ maps an $a$-edge to a $c$-edge (for otherwise an earlier argument would apply after interchanging $b$ and~$c$). Since $a$-edges can be mapped to $b$-edges, this implies that no $b$-edge can be mapped to a $c$-edge. Therefore, every automorphism of $BX$ is an automorphism of $BX_0$, where $X_0 = \Cay(\ZZ_n,\{\pm a, \pm b\})$. Let $X_0'$ be the connected component of~$X_0$ that contains~$0$. Recalling that $|a|=4$, we see that if $|b|=3$ then 
	$ X_0'$ is isomorphic to $\Cay(\ZZ_{12},\{ \pm 3, \pm 4 \}) $,
and if $|b|=6$, then
	$ X_0'$ is isomorphic to $\Cay(\ZZ_{12},\{ \pm 2, \pm 3 \})$.
Both of these graphs are stable (by \fullcref{small}{12} or \cref{val4}). So \cref{StableSubgraph} tells us that $X$ is stable as well, a contradiction.

\begin{subcase}
Assume every automorphism of $BX$ maps $a$-edges to $a$-edges.
\end{subcase}
Then every automorphism of~$BX$ is also an automorphism of $BX_0$, where $X_0 = \Cay(\ZZ_n, \{\pm b, \pm c\})$. As usual, let $X_0'$ be the connected component of $X_0$ that contains~$0$. 

\begin{subsubcase}
Assume $X_0$ is not twin-free.
\end{subsubcase}
Then $b = c + \nh$ (perhaps after replacing $c$ with~$-c$). Since $a + \nh = -a$, this implies $S + \nh = S$, which contradicts the fact that $X$ is twin-free.

\begin{subsubcase}
Assume  $X_0$ is not connected \textup(but is twin-free\textup).
\end{subsubcase}
We know that $X_0'$ is not stable (for otherwise \cref{StableSubgraph} contradicts the fact that $X$ is not stable). Since $X_0'$ is connected (by definition) and twin-free (by assumption), this implies that it has even order (see \cref{OddCirculant}). Since $\langle a, b, c \rangle = \ZZ_n$ and $|a| = 4$, we conclude that  $\langle b, c \rangle = 2 \ZZ_n$ and $n \equiv 4 \pmod{8}$.

\smallbreak

$\bullet$ If $X_0'$ is bipartite, then $BX_0$ is isomorphic to the union of four disjoint copies of~$X_0'$. Since $|a| = 4$ and $BX$ is connected, this implies that $BX \cong C_4 \cartprod X_0'$. 
Since $X_0'$ has even valency, and $|V(X_0')|/2$ is odd, it it is easy to see that $X_0'$ does not have $K_2$ as a Cartesian factor. (If $X_0' \cong K_2 \cartprod Y$, then $Y$ is a regular graph of odd valency and odd order, which is impossible.) This implies (by \cref{Aut(C4timesX)}) that
	\[ |{\Aut BX}| = |{\Aut (C_4 \cartprod X_0')}|= |{\Aut C_4}| \cdot |{\Aut X_0' }|  = 8 \, |{\Aut X_0' }| . \]
Also note that, since $X_0'$ is a bipartite circulant graph whose order is congruent to~$2$ modulo~$4$, it is (isomorphic to) the canonical double cover of a circulant graph of odd order. Since connected, twin-free, circulant graphs of odd order are stable (see \cref{OddCirculant}) and $2a$ is the element of order~$2$ in~$\ZZ_n$, we conclude that if $\beta$ is any automorphism of~$X_0'$, then $\beta(v + 2a) = \beta(v) + 2a$ for every $v \in V(X_0')$. This implies that we can extend~$\beta$ to an automorphism~$\beta'$ of~$X$ by defining $\beta'(a + v) = a + \beta(v)$ for $v \in V(X_0')$; so $\Aut X$ contains a copy of $\Aut X_0'$. Since $\Aut X$ also contains the translation $v \mapsto v + a$ and the negation automorphism $v \mapsto -v$, we conclude that 
	$|{\Aut X}| \ge 4 \, |{\Aut X_0' }|$.
Combining this with the above calculation of $|{\Aut BX}|$ contradicts the fact that $X$ is not stable.

\smallbreak

$\bullet$ If $X_0'$ is not bipartite, then $X_0'$ is nontrivially unstable, and therefore must be described by \cref{val4}. Since $|2\ZZ_n| \equiv 2 \pmod{4}$, we must be in the situation of \fullref{val4}{k2=1}: $\langle c \rangle = 2 \ZZ_n$, and $b = mc + \nh$, for some $m \in \ZZ_n^\times$, such that $m^2 \equiv \pm 1 \pmod{n}$. (Technically, \cref{val4} only tells us that $m^2 \equiv 1 \pmod{n/2}$. However, we know that $m$ is odd, so $m^2 \equiv 1 \pmod{4}$. Therefore $m^2 \equiv 1 \pmod{n}$.) We also have $ma + \nh = \pm a$ (since $|a| = 4$). Therefore $m S + \nh = S$, which means that $X$ has Wilson type~\pref{Wilson-C4}.

\begin{subsubcase}
Assume $X_0$ is connected and bipartite.
\end{subsubcase}
 This means that $b$ and~$c$ are odd, so $S \cap 2\ZZ_n = \{\pm a\}$. Since $|a| = 4$, we conclude that $(S \cap 2\ZZ_n) + \nh = (S \cap 2\ZZ_n)$, so $X$ has Wilson type~\pref{Wilson-C1}.

\begin{subsubcase}
Assume $X_0$ is nontrivially unstable.
\end{subsubcase}
 Then \cref{val4} tells us that $X_0$ has Wilson type~\pref{Wilson-C4}. Then $X$ also has Wilson type~\pref{Wilson-C4}, with the same value of~$m$.

\begin{case} \label{val6pf-neither}
Assume neither of the previous \lcnamecref{val6pf-2a=2b}s apply \textup(even after permuting and/or negating some of the generators\textup).
\end{case}
Let $\alpha$ be an automorphism of $BX$ that fixes $(0,0)$.
The assumption of this \lcnamecref{val6pf-neither} implies that $2s \neq 2t$ for all $s,t \in S$, such that $s \neq t$. Therefore, \cref{2S'-BX} implies that the cosets of $2\ZZ_n \times \{0\}$ are blocks for the action of $\Aut BX$. So $\alpha$ must fix the two cosets that are in $\ZZ_n \times \{0\}$, and either fixes or interchanges the other two.  However, also note that $S$ is the disjoint union of 
	\[ \text{$S_e \coloneqq S \cap 2\ZZ_n$ and $S_o \coloneqq S \cap (1 + 2 \ZZ_n)$.} \]
Each of these two sets has even cardinality (since it is closed under inverses), and $|S_e| + |S_o| = 6$, so it is easy to see that $|S_e| \neq |S_o|$. Therefore, $\alpha$ cannot interchange $S_e$ and~$S_o$, which means that $\alpha$ must fix all four cosets of $2\ZZ_n \times \{0\}$. So 
	\begin{align} \label{val6pf-neither-edges}
	\text{$\alpha$ maps $S_e$-edges to $S_e$-edges, and maps $S_o$-edges to $S_o$-edges.}
	\end{align}
Hence, by \cref{StableSubgraph}, we know that 
	\[ \text{the connected components of $\Cay(\ZZ_n, S_e)$ are unstable.} \]
(The connected components of $\Cay(\ZZ_n, S_o)$ are always unstable, since they are bipartite.) We also know that $\Cay(\ZZ_n, S_e)$ is twin-free (since $X$ does not have Wilson type~\pref{Wilson-C1}). Therefore, we see from \cref{OddCirculant} that $|\langle S_e \rangle|$ is even, so
	\[ n \equiv 0 \pmod{4} . \]

From \cref{NormalCover}, we see that $BX$ is not normal. So applying \cref{4cyclenormal} to $BX$ implies that  
	either \ $|a| = 4$ \ or \ $3a = b$ \ or \ $2a = 2b$ \ or \ $2a = b + c$
 (perhaps after permuting and/or negating some of the generators).
Since neither of the previous cases apply, this implies that 
	\[ \text{either $3a = b$ or $2a = b + c$.} \]
We will consider each of these two possibilities as a separate \lcnamecref{val6pf-neither-3a=3b}.

Recall that \fullcref{CayleyDefn}{S1} introduced $\tilde s$ as an abbreviation for $(s,1)$ with $s\in S$.

\begin{subcase} \label{val6pf-neither-3a=3b}
Assume $3a = b$.
\end{subcase}
Then also $3\tilde a = \tilde b$. 

We claim that $|\tilde a| \ge 10$. First of all, we have $|\tilde a| \neq 2$, because $|a| \notin \{1,2\}$. 
We also have $|\tilde a| \neq 4$, because $|a| \neq 4$ (since \cref{val6pf-a=4} does not apply). 
Now, note that if $|\tilde a|=6$, then $|a| = 3$ or $|a| = 6$; however, the fact that $b = 3a$ would then imply that $|a| = 3|b|$, so $|b| \in \{1,2\}$, which is a contradiction. 
Finally, note that if $|a|=8$, then $2(-a) = -2a = 6a = 2(3a) = 2b$, which contradicts the assumption that \cref{val6pf-2a=2b} does not apply.
This completes the proof of the claim.

\begin{subsubcase}
Assume $|\tilde a|=10$. 
\end{subsubcase}
Then $|a|=5$ or $|a|=10$. However, if $|a|=5$, then $|b|=5$ as well and they generate the same cyclic subgroup of $\ZZ_n$. In particular, the connected components of $\Cay(\ZZ_n,\{\pm a, \pm b\})$ are isomorphic to $K_5$, which is stable, so by \cref{StableSubgraph}, we get that $X$ is stable, a contradiction.

Therefore, we must have $|a|=10$. (Then $|b|=10$ and they generate the same cyclic subgroup.)
Since $n \equiv 0 \pmod{4}$, we may write 
	\[ \text{$n = 20k$ for some $k \in \ZZ$.} \]
Since $|a| = |b| = 10 \not\equiv 0 \pmod{4}$ and $n \equiv 0 \pmod{4}$, it is clear that $|c|$ is even (in fact, it is divisible by~$4$). It follows that $|c|$ is either $n$ (if $|c|$ is divisible by~$5$) or $n/5$ (otherwise). So we see that (up to isomorphism) $X$ is
	\[ \text{$\Cay(\ZZ_{20k},\{\pm c,2k,6k,14k,18k\})$ \ with $c = 1$ or $c = 5$.} \]
From~\pref{val6pf-neither-edges}, we know that $\alpha$ is also an automorphism of the graphs 
	\[ Y_1 = \Cay(\ZZ_{20k}\times \ZZ_2,\{(\pm c, 1)\}) \] 
and 
	\[ Y_2 = \Cay(\ZZ_{20k}\times \ZZ_2,\{(2k,1),(6k,1),(14k,1),(18k,1)\}) . \]

Note that $(0,0)$ and $(10k,1)$ lie in the same connected component of $Y_2$, which is isomorphic to $K_{5,5}-5K_2$. In this component, $(10k,1)$ is the unique vertex at distance $3$ from $(0,0)$, so $\alpha$ fixes $(10k,1)$. 

Also note that the vertices $(0,1)$ and $(10k,1)$ lie in the same connected component of~$Y_1$, which is a cycle (of length $20k$ or $4k$), and that these two vertices are diametrically opposite on this cycle. Since we already know that $\alpha$ fixes $(10k,1)$, it must also fix $(0,1)$. We conclude that $X$ is stable, a contradiction.

\begin{subsubcase}
Assume $|\tilde a|=12$.
\end{subsubcase}
Then $|a|=12$. Since $b = 3a$, we have $|b|=|3a|=4$, which contradicts the fact that \cref{val6pf-a=4} does not apply.

\begin{subsubcase} \label{val6pf-neither-3a=3b-a>14}
Assume $|\tilde a| \ge 14$.
\end{subsubcase}
Let $\mathsf{B}_2$ be the subgraph induced by the ball of radius~$2$ centered at~$0$ in $\Cay(\ZZ_n\times \ZZ_2, \{\pm \tilde a, \pm \tilde b\})$. 
This graph is drawn in \cref{B2}, under the assumption of this \lcnamecref{val6pf-neither-3a=3b-a>14} that $|\tilde a| \ge 14$. From this drawing, it can be seen that $\pm \tilde b$ are the only vertices in~$\mathsf{B}_2$ that have a pendant edge. (These edges are colored white in the figure.) So $\{\pm \tilde b\}$ is $\alpha$-invariant. This means that $\alpha$ maps $b$-edges to $b$-edges. Since we already know that $\alpha$ maps $c$-edges to $c$-edges, it must also map $a$-edges to $a$-edges. 

\begin{figure}[htbp]
\centering
\includegraphics[scale=1.1]{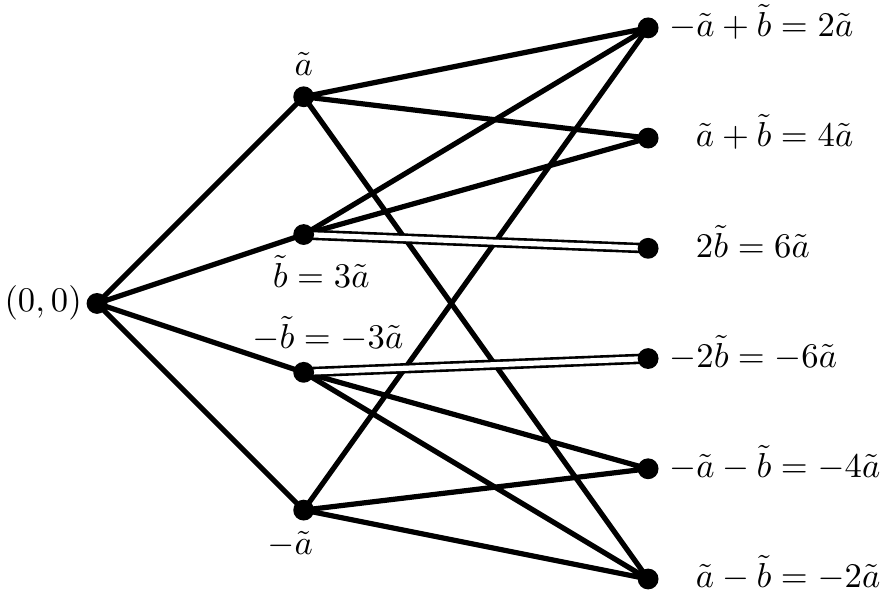}
\caption{The subgraph $\mathsf{B}_2$ induced by the ball of radius $2$ centered at $(0,0)$ in $\Cay(\ZZ_n \times \ZZ_2, \pm \tilde a,\pm \tilde b)$.}
\label{B2}
\end{figure}
%
%
%
%
%
%
%
%
%

In the terminology of~\cite{HujdurovicEtal-CCA}, this means that $\alpha$~is a \emph{color-preserving} graph automorphism. We will now use a simple argument from \cite[\S4]{HujdurovicEtal-CCA} to establish that $\alpha$ is a group automorphism (so $BX$ is normal, and then it follows from \cref{NormalCover} that $X$ has Wilson type~\pref{Wilson-C4}).

We provide only a sketch of the proof.
By composing with negation, if necessary, we may assume that $\alpha(\tilde a) = \tilde a$. This implies $\alpha(k \tilde a) = k \tilde a$ for all $k \in \ZZ$. Let $m \in \ZZ^+$, such that $m \tilde c \in \langle \tilde a \rangle$. Then $\alpha(k m \tilde c) = k m \tilde c$ for all $k \in \ZZ$. If $2 m \tilde c \neq 0$, this implies that $\alpha(\ell \tilde c) = \ell \tilde c$ for all $\ell \in \ZZ$; in fact, $\alpha(k \tilde a + \ell \tilde c) = k \tilde a + \ell \tilde c$, for all $k,\ell \in \ZZ$, which means that $\alpha$ is the identity map, contradicting the fact that $\alpha \notin \Aut X \times \ZZ_2$.

Therefore, we may assume $2 m \tilde c = 0$, for all $m \in \ZZ$, such that $m \tilde c \in \langle \tilde a \rangle$. This means that $|\langle \tilde c \rangle \cap \langle \tilde a \rangle| \le 2$, so there is a group automorphism of $\ZZ_n$ that fixes~$\tilde a$ and negates~$\tilde c$. So we may assume that $\alpha(\tilde c) = \tilde c$. Since $\tilde a + \tilde c$ is the only common neighbor of $\tilde a$ and~$\tilde c$, we must have $\alpha(\tilde a + \tilde c) = \tilde a + \tilde c$. Similarly, we must have $\alpha(2 \tilde a + \tilde c) = 2\tilde a + \tilde c$ and $\alpha(\tilde a + 2\tilde c) = \tilde a + 2\tilde c$. Repeating the argument shows that $\alpha(k \tilde a + \ell \tilde c) = k \tilde a + \ell \tilde c$, for all $k,\ell \in \ZZ$, so, once again, $\alpha$ is the identity map.

\begin{subcase} \label{val6pf-neither-2a=b+c}
Assume $2a = b + c$.
\end{subcase}
(Note that this implies $2\tilde a = \tilde b + \tilde c$.)
We will show that $\alpha(k \tilde b + \ell \tilde a) = k \alpha(\tilde b) + \ell \alpha(\tilde a)$, for all $k, \ell \in \ZZ^{\ge 0}$. (This implies that $\alpha$ is a group automorphism of~$\ZZ_n$, so $BX$ is normal, so \cref{NormalCover} implies that $X$ has Wilson type~\pref{Wilson-C4}.)

Since $b$ and~$c$ have the same parity, we see from \pref{val6pf-neither-edges} that $\alpha$ maps $\{b,c\}$-edges to $\{b,c\}$-edges, and maps $a$-edges to $a$-edges. In particular, we may assume (by composing with negation if necessary) that 
	\begin{align} \label{val6pf-neither-2a=b+c-fixa}
	\text{$\alpha$ fixes every element of~$\langle \tilde a \rangle$.}
	\end{align}

Since $\alpha$ maps $\{b,c\}$-edges to $\{b,c\}$-edges, $\alpha$ is an automorphism of the graph $\Cay(\ZZ_n\times\ZZ_2,\{\pm \tilde{b},\pm \tilde{c}\})$, which has at most two connected components. Let $X_0'$ be its connected component containing $(0,0)$ with the vertex set $\langle\tilde b,\tilde c\rangle$. Since $\alpha$ fixes $(0,0)$, it restricts to an automorphism of $X_0'$. Because we are assuming \cref{val6pf-2a=2b} and \cref{val6pf-a=4} do not hold and we can additionally assume \cref{val6pf-neither-3a=3b} does not hold either, \cref{4cyclenormal} applies to $X_0'$. It follows that the restriction of $\alpha$ to $V(X_0')$ is a group automorphism of $\langle \tilde b, \tilde c \rangle$. We let $b', c' \in \{ \pm b, \pm c \}$, such that $\alpha(\tilde b) = \tilde{b'}$ and $\alpha(\tilde c) = \tilde{c'}$. It follows that:
\begin{align}\label{val6pf-2a=b+c-normal}
\alpha(k\tilde b + \ell\tilde c) = k\tilde{b'}+\ell\tilde{c'} \qquad \forall k,\ell\in\ZZ
\end{align}

Notice that:
	\begin{align*}
	\tilde b + \tilde c &= 2\tilde a
		&& (\text{assumption of \cref{val6pf-neither-2a=b+c}})
	 \\&= \alpha(2\tilde a) 
	 	&& \pref{val6pf-neither-2a=b+c-fixa}
	\\&= \alpha(\tilde b+ \tilde c)
		&& \text{($2\tilde a = \tilde b+ \tilde c$)}
	\\&= \tilde b' + \tilde c' 
		&& \pref{val6pf-2a=b+c-normal}
	. \end{align*}

As we have already established that the hypothesis of \cref{4cyclenormal} holds, we obtain that $\{b', c'\} = \{b,c\}$.

To complete the proof of this \lcnamecref{val6pf-neither-2a=b+c}, we now prove by induction on~$k$ that, for all $k,\ell \in \ZZ^{\ge 0}$, we have
	\[ \alpha(k \tilde b + \ell \tilde a) = k \tilde b' + \ell \tilde a. \]
The base case is provided by~\pref{val6pf-neither-2a=b+c-fixa}, so assume $k > 0$. 
Since $\alpha$ maps $a$-edges to $a$-edges, there exists $\epsilon \in \{\pm 1\}$, such that $\alpha(k \tilde b + \ell \tilde a) = k \tilde b' + \epsilon \ell \tilde a$ for all $\ell \in \ZZ$. We wish to show $\epsilon = 1$, so suppose $\epsilon = -1$. (This will lead to a contradiction.) Letting $\ell = -2$ tells us 
	\begin{align*}
	(k -1) \tilde b' - \tilde c' 
	&= \alpha \bigl( (k -1) \tilde b - \tilde c \bigr) 
		&& \text{\pref{val6pf-2a=b+c-normal}}
	 \\&= \alpha(k \tilde b - 2\tilde a) 
	 	&& \text{($2 \tilde a = \tilde b + \tilde c$)}
	\\&= k \tilde b' + 2 \tilde a
		&& (\epsilon = -1)
	\\&= k \tilde b' + (\tilde b + \tilde c) 
	 	&& \text{($2 \tilde a = \tilde b + \tilde c$)}
	\\&= k \tilde b' + (\tilde b' + \tilde c') 
	 	&& (\{b, c\} = \{b', c'\})
	. \end{align*}
This implies $-2\tilde b' = 2 \tilde c'$, which contradicts the fact that \cref{val6pf-2a=2b} does not apply. 
\end{proof}

The following result provides a more explicit version of \cref{val6}. \Cref{val6nontrivial} explains which of these graphs are nontrivially unstable.

\begin{cor} \label{explicit6}
A circulant graph $X = \Cay(\ZZ_n, \{\pm a, \pm b, \pm c \})$ of valency\/~$6$ is unstable if and only if either it is trivially unstable, or it is one of the following:
\noprelistbreak
	\begin{enumerate}
	
	\setcounter{enumi}{0}
	
		\item \label{explicit6-C1Se2}
		$\Cay(\ZZ_{8k},\{\pm a, \pm b, \pm 2k\})$, where $a$ and~$b$ are odd,  which is of Wilson type~\pref{Wilson-C1}.
	
		\item \label{explicit6-C1Se4}
		$\Cay(\ZZ_{4k},\{\pm a, \pm b, \pm b + 2k\})$, where $a$ is odd and $b$ is even, which is of Wilson type~\pref{Wilson-C1}.
		
		\item \label{explicit6-C2}
		$\Cay \bigl( \ZZ_{4k}, \bigl\{ \pm a, \pm (a + k), \pm(a - k) \bigr\} \bigr)$, where $a \equiv 0 \pmod{4}$ and $k$ is odd,  which is of Wilson type~\pref{Wilson-C2}.
		
 		\item \label{explicit6-C3order2}
                $\Cay(\ZZ_{8k},\{\pm a, \pm b, \pm b + 4k\})$, where $a$ is even and $|a|$ is divisible by~$4$, which is of Wilson type~\pref{Wilson-C3}.

        \item \label{explicit6-C3order4}
                $\Cay(\ZZ_{8k},\{\pm a, \pm k, \pm 3k\})$, where $a \equiv 0 \pmod{4}$ and $k$ is odd, which is of Wilson type~\pref{Wilson-C3}.

		\item \label{explicit6-C4}
		$\Cay(\ZZ_{4k},\{\pm a, \pm b, \pm m b + 2k \})$, where 
		\[ \text{$\gcd(m, 4k) = 1$, \quad $(m-1)a \equiv 2k \pmod{4k}$, \quad and} \]
		\[ \text{either \ 
			$m^2 \equiv 1 \pmod{4k}$
			\ or \ 
			$(m^2 + 1) b \equiv 0 \pmod{4k}$,} \]
		  which is of Wilson type~\pref{Wilson-C4}.

		 \item \label{explicit6-C4more}
           $\Cay(\ZZ_{8k},\{\pm a, \pm b, \pm c \})$, where there exists $m \in \ZZ$, such that
           \[ \text{$\gcd(m, 8k) = 1$, \quad $m^2 \equiv 1 \pmod{8k}$, \quad and} \]
        \[ (m-1)a \equiv (m + 1)b \equiv (m + 1)c \equiv 4k \pmod{8k} , \]
     which is of Wilson type~\pref{Wilson-C4}.
		  
	\end{enumerate}
\end{cor}

\begin{proof}
($\Leftarrow$) It is easy to see that each graph has the specified Wilson type, and is therefore unstable.
(The arguments in the other direction of the proof can provide some hints, if required.)

($\Rightarrow$) If $X$ is unstable, then we know from \cref{val6} that $X$ has Wilson type~\pref{Wilson-C1}, \pref{Wilson-C2}, \pref{Wilson-C3}, or~\pref{Wilson-C4}. We treat each of these possibilities as a separate case.

\refstepcounter{caseholder} 

\begin{case}
Assume $X$ has Wilson type~\pref{Wilson-C1}.
\end{case}

\begin{subcase}
Assume $|S_e| = 2$.
\end{subcase}
Then we may assume $S_e = \{\pm c \}$. Since $X$ has Wilson type~\pref{Wilson-C1}, we must have $-c = c + \nh$, so $|c| = 4$. (Since $c \in S_e$, this implies that $n$ is divisible by~$8$.) Therefore $X$ is as described in part~\pref{explicit6-C1Se2} of the \lcnamecref{explicit6}.

\begin{subcase}
Assume $|S_e| = 4$.
\end{subcase}
Then we may assume $S_e = \{\pm b, \pm c\}$. Since $X$ has Wilson type~\pref{Wilson-C1} (and $|S_e|$ is a power of~$2$), we must have $S_e + \nh = S_e$. Therefore, we may assume $c = b + \nh$ (because we cannot have $b + \nh = -b$ \emph{and} $c + \nh = -c$). Since $b$ and~$c$ are even, this implies $\nh$ is even, so $n$ is divisible by~$4$. Hence, $X$ is as described in part~\pref{explicit6-C1Se4} of the \lcnamecref{explicit6}.

\begin{case}
Assume $X$ has Wilson type~\pref{Wilson-C2}.
\end{case}
Let $h \in 1 + 2 \ZZ_n$, such that \pref{Wilson-C2-So} and~\pref{Wilson-C2-s+b} of condition~\pref{Wilson-C2} hold.

\begin{subcase}
Assume $|S_o| = 2$.
\end{subcase}
Then we may assume $S_o = \{\pm c\}$, so part~\pref{Wilson-C2-So} of~\pref{Wilson-C2} tells us $-c = c + 2h$ and $c = -c + 2h$. This implies that $|c| = 4$ and $|h| = 4$. So we may assume $c = h = n/4$. Then it is obvious that $-c \equiv -h \pmod{4}$, so part~\pref{Wilson-C2-s+b} of~\pref{Wilson-C2} implies that $0 = -c + h \in S$. This contradicts the fact that graphs in this paper do not have loops (see~\cref{GraphsAreSimple}).

\begin{subcase}
Assume $|S_o| = 4$.
\end{subcase}
Then we may assume $S_o = \{\pm b, \pm c\}$, so part~\pref{Wilson-C2-So} of~\pref{Wilson-C2} implies that we may assume $c = b + 2h$ (by replacing $b$ with~$-b$ if necessary), and we must have $|2h| \in \{2,4\}$.

\begin{subsubcase}
Assume $|2h| = 2$.
\end{subsubcase}
Then we have $c = b + \nh$. We may assume $b \equiv -h \pmod{4}$ (by interchanging $b$ and~$c$ if necessary). So part~\pref{Wilson-C2-So} of condition~\pref{Wilson-C2} implies $b + h \in S$, which means $a = b + h$ (perhaps after replacing $a$ with~$-a$). In other words, we have $b = a - h$; then $c = b + 2h = a + h$.  Therefore, since $|h| = 4$, we have 
	\[ S = \bigl\{ \pm a, \pm \bigl( a + (n/4) \bigr), \pm \bigl( a - (n/4) \bigr) \bigr\} . \] 
We see from part~\pref{Wilson-C2-s+b} of condition~\pref{Wilson-C2} that $S$ contains an element that is divisible by~$4$, so we must have $a \equiv 0 \pmod{4}$.
Also, since $a + (n/4) = b$ is odd, we know that $n/4$ is odd.
Hence, $X$ is as described in part~\pref{explicit6-C2} of the \lcnamecref{explicit6}.

\begin{subsubcase} \label{explicit6pf-C2-4-4}
Assume $|2h| = 4$.
\end{subsubcase}
Then 
	\[  \{\pm b, \pm c\} = \{b , b + 2h, b + 4h, b + 6h \}, \]
so $-c = b + 4h$ (and $-b = b + 6h$). Since $c = b + 2h$, this implies $2c2= 6h$ has order~$4$, so $|b| = 8$. So $S_o = \{\pm n/8, \pm 3n/8 \}$ consists of all of the elements of order~$8$. Since $-h$ has order~$8$, we conclude that $-h \in S_o$, so part~\pref{Wilson-C2-s+b} of condition~\pref{Wilson-C2} implies $0 \in S$, which (again) contradicts the fact that graphs in this paper do not have loops.

\begin{case}
Assume $X$ has Wilson type~\pref{Wilson-C3}.
\end{case}
Let $H$, $R$, and~$d$ be as in the definition of Wilson type~\pref{Wilson-C3}. We may assume that $X$ does not have Wilson type~\pref{Wilson-C1}, so $R$ contains at least one element of~$S_e$; for definiteness, let us say that $a$ is in $R \cap S_e$. Since $r/d$ is odd for every $r \in R$, we know that all elements of~$R$ have the same parity, so this implies $R \subseteq S_e$. 

\begin{subcase}
Assume $|R| = 2$.
\end{subcase}
This means $R = \{\pm a\}$, $d = \gcd(n,a)$, and $\{\pm b, \pm c\} + H = \{\pm b, \pm c\}$. 
Let $h$ be a generator of~$H$, so $b + h = c$ (perhaps after replacing $c$ with~$-c$). 

\begin{subsubcase}
Assume $|h| = 2$ \textup(so $h = \nh$\textup).
\end{subsubcase}
Since $n/d$ is even, we know that $\nh \in d \ZZ_n$. Therefore, the last condition in~\pref{Wilson-C3} implies that $|a|$ is divisible by~$4$.
(Since $a$ is even, this implies that $n$ is divisible by~$8$.)
Hence, $X$ is as described in part~\pref{explicit6-C3order2} of the \lcnamecref{explicit6}.

\begin{subsubcase}
Assume $|h| = 4$.
\end{subsubcase}
The argument of \cref{explicit6pf-C2-4-4} shows $\{\pm b, \pm c\} = \{\pm n/8, \pm 3n/8\}$. Since the elements of~$S$ cannot all be even, we know that $n/8$ is odd. Then, since $a$ is even, we know that $2a \ZZ_n$ does not contain an element of order~$4$, so we conclude from the last sentence of condition~\pref{Wilson-C3} that $a \ZZ_n$ does not contain~$H$. This means that $a$ is divisible by~$4$.
Hence, $X$ is as described in part~\pref{explicit6-C3order4} of the \lcnamecref{explicit6}.

\begin{subcase}
Assume $|R| = 4$.
\end{subcase}
This means we may assume $R = \{\pm a, \pm b\}$, $H = \langle \nh \rangle$, and $|c| = 4$.
Since $n/d$ is even, we know that the elements of~$R$ have even order, so $\nh \in \langle R \rangle = d \ZZ_n$. Therefore, the last sentence of condition~\pref{Wilson-C3} implies $\nh \in 2 d \ZZ_n$, which means $n/d$ is divisible by~$4$. Since $r/d$ is odd for every $r \in R$, this implies $|a| \equiv |b| \equiv 0 \pmod{4}$. Since $a \in R \subseteq S_e$, we conclude that $n$ is divisible by~$8$. So $c = \pm n/4$ is even. This contradicts the fact that at least one element of~$S$ must be odd (since $S$ generates~$\ZZ_n$).

\begin{case}
Assume $X$ has Wilson type~\pref{Wilson-C4}.
\end{case}
This means there exists $m \in \ZZ$ such that $\nh + m S = S$ and $\gcd(m, n) = 1$. Let $\alpha(x) \coloneqq \nh + m x$, so $\alpha(S) = S$.
Since $|S_e| \neq |S_o|$, we know that $\alpha(S_e) \neq S_o$. Therefore, we must have $\alpha(S_e) = S_e$ and $\alpha(S_o) = S_o$. This means that $\nh$ is even (so $n$ is divisible by~$4$).  

Assume, without loss of generality, that $b$ and~$c$ have the same parity (and $a$ has the opposite parity). 

We must have $\nh + m  a \in \{\pm a\}$. So we may assume $\nh + m  a = a$ (by replacing $m $ with~$-m $ if necessary), so $(m -1)a = \nh$. 

We also have $\alpha(\{ \pm b, \pm c \}) = \{ \pm b, \pm c \}$.
Therefore, either $\alpha$ fixes $\{\pm b \}$ and~$\{\pm c\}$,
    or $\alpha$ interchanges these two sets.

    \begin{subcase}
    Assume $\alpha(b) \in \{ \pm b \}$ and $\alpha(c) \in \{ \pm c \}$.
    \end{subcase}
    This means $\nh + m b = \pm b$ and $\nh + m c = \pm c$, so
    $(m \pm 1) b = \nh$ and $(m \pm 1) c = \nh$. Since
    $(m - 1) a = \nh$, and $a$ has the opposite parity from $b$
    and~$c$, we cannot have $(m - 1) b = \nh$ or $(m - 1)c = \nh$.
    Therefore, we must have $(m + 1) b = \nh$ and $(m + 1) c = \nh$.
    
    We know that $m$ is odd, so either $m - 1$ or $m + 1$ is divisible by~$4$. Since
    $(m - 1)a = \nh$ and $(m + 1)b = \nh$, this implies that $\nh$ is divisible by~$4$.
    So $n$ is divisible by~$8$

    Note that, since $m + 1$ is even, we have
        \[ (m^2 - 1)a = (m + 1)(m - 1)a = (m + 1) \nh = 0 . \]
    Similarly, we also have $(m^2 - 1)b = 0$ and $(m^2 - 1)c = 0$.
    Since $\langle a, b, c \rangle = \ZZ_n$, this implies that
    $m^2 \equiv 1 \pmod{n}$. Hence, $X$ is as described in
    part~\pref{explicit6-C4more} of the \lcnamecref{explicit6}.

    \begin{subcase}
    Assume $\alpha(b) \in \{ \pm c \}$ and $\alpha(c) \in \{ \pm b \}$.
    \end{subcase}
    We have $\alpha^2(b) \in \{\pm b\}$, which means
	$\pm b = \nh + m (\nh + m b) = m ^2 b$,
so there exists $\epsilon \in \{\pm 1\}$, such that $(m ^2 + \epsilon) b = 0$.

\begin{subsubcase}
Assume $\epsilon = -1$.
\end{subsubcase}
This means $(m ^2 - 1) b = 0$. Since $m $ must be odd, we also have 
	\[(m ^2 - 1)a = (m  + 1) \cdot (m -1) a = (m  + 1) \nh  = 0 \]
and $(m ^2 - 1) \nh = 0$. Since 
	\[ \ZZ_n = \langle a, b, c \rangle = \langle a, b, \nh + m b \rangle = \langle a, b, \nh \rangle ,\]
we conclude that $(m ^2 - 1) \ZZ_n = \{0\}$, so $m ^2 \equiv 1 \pmod{n}$.
Hence, $X$ is as described in part~\pref{explicit6-C4} of the \lcnamecref{explicit6} (with $m^2 \equiv 1 \pmod{n}$).

\begin{subsubcase}
Assume $\epsilon = 1$.
\end{subsubcase}
This means $(m ^2 + 1) b = 0$. 
Hence, $X$ is as described in part~\pref{explicit6-C4} of the \lcnamecref{explicit6} (with $(m^2 + 1) b = 0$).
\end{proof}

\begin{rem} \label{val6nontrivial}
It is easy to determine whether a graph in \cref{explicit6} is nontrivially unstable. Indeed, here are quite simple necessary and sufficient conditions for each of the lists in the statement of the \lcnamecref{explicit6}:
\begin{enumerate}
        \item[\fullref{explicit6}{C1Se2}:]
        $\gcd(a, b, k) = 1$ and $b \notin \{\pm a + 4k\}$.

        \item[\fullref{explicit6}{C1Se4}:]
        $\gcd(a, b, k) = 1$ and $a \notin \{\pm k\}$.

        \item[\fullref{explicit6}{C2}:]
        $\gcd(a, k) = 1$.

        \item[\fullref{explicit6}{C3order2}:]
        $\gcd(a, b, 4k) = 1$ and $a \notin \{\pm 2k\}$.

        \item[\fullref{explicit6}{C3order4}:]
        $\gcd(a, k) = 1$.
        
        \item[\fullref{explicit6}{C4}:]
        $\gcd(a, b, 2k) = 1$, either $a$ or~$b$ is even,
                        and either $a \notin \{\pm k\}$ or $mb \notin \{\pm b\}$.

        \item[\fullref{explicit6}{C4more}:]
        $\gcd(a, b, c, 4k) = 1$, either $a$ or~$b$ is even,
                        and either $a \notin \{\pm 2k\}$ or $c \notin \{\pm b + 4k\}$.

\end{enumerate}
\end{rem}

\begin{proof}
For convenience, let $S = \{\pm a, \pm b, \pm c\}$.

It is clear that $X$ is connected if and only if $\gcd(S \cup \{n\}) = 1$. Therefore, the first condition in each part of the \lcnamecref{val6nontrivial} is precisely the condition for~$X$ to be connected.

Knowing that $\gcd(S \cup \{n\}) = 1$ implies that at least one element of~$S$ is odd (since $n$ is even). Therefore, $X$ is nonbipartite if and only if at least one element of~$S$ is even. It is obvious that $S$ has an even element in all parts of \cref{explicit6} other than~\pref{explicit6-C4} and~\pref{explicit6-C4more}, so the statement of the \lcnamecref{val6nontrivial} only adds this as an explicit condition for parts~\pref{explicit6-C4} and~\pref{explicit6-C4more}. (In part~\pref{explicit6-C4more}, the fact that $(m + 1)b \equiv (m + 1)c \equiv 4k \pmod{8k}$ implies that $b$ and~$c$ have the same parity, so there is no need to mention the possibility that $c$ is even.)

Now, let us suppose that $X$ is not twin-free (but is connected and nonbipartite). Then by \fullcref{nottwinfree}{Y}, it follows that $X\cong Y\wr \overline{K_m}$ with $Y$ being a connected circulant of valency $\delta$ and $m\geq 2$ an integer. Clearly, $6=\delta m$, so $m\in \{2,3,6\}$.
        \begin{itemize}
        \item If $m=6$, then $\delta=1$ and $Y\cong K_2$. Consequently, $X\cong K_2\wr \overline{K_6}\cong K_{6,6}$, which contradicts the fact that $X$ is not bipartite.
        \item If $m=3$, then $\delta=2$ and $Y$ is a cycle of even length, which again contradicts the fact that $X$ is not bipartite.
        \item If $m=2$, then from the proof of \fullcref{nottwinfree}{d=4}, it can be concluded that the unique twin of $0$ is~$\nh$. Therefore, it must hold that $\nh + S = S$.
        \end{itemize}
Thus, we see that $X$ is twin-free if and only if $\nh + S \neq S$.

In parts~\pref{explicit6-C2} and~\pref{explicit6-C3order4}, it is clear that if $\nh + S = S$, then $a + \nh = -a$, which means $a = \pm n/4$.  Since $a \equiv 0 \pmod{4}$, this implies $n$ is divisible by~$16$, which contradicts the fact that $k$ is odd. So all of the graphs of these two types are twin-free. All other parts of the \lcnamecref{val6nontrivial} add a final condition that specifically rules out the possibility that $\nh + S = S$.
\end{proof}

\section{Unstable circulants of valency \texorpdfstring{$7$}{7}} \label{val7Sect}

\begin{thm} \label{val7}
A circulant graph\/ $\Cay(\ZZ_n, S)$ of valency\/~$7$ is unstable if and only if either it is trivially unstable, or it is one of the following:
\noprelistbreak
\begin{enumerate}

\item \label{val7-Z6}
$\Cay(\ZZ_{6k},\{\pm 2t, \pm2(k - t), \pm2(k + t), 3k\})$, with $k$~odd, which has Wilson type \pref{Wilson-C1}.

\item \label{val7-bc-odd}
$\Cay(\ZZ_{12k},\{\pm 2k,\pm b,\pm c, 6k\})$, with $b$ and $c$ odd, which has Wilson type \pref{Wilson-C1}.

\item \label{val7-Z20}
$\Cay(\ZZ_{20k},\{\pm t,\pm 2k,\pm 6k,10k\})$, with $t$~odd, which has Wilson type \pref{Wilson-C1}.

\item \label{val7-Z4}
$\Cay(\ZZ_{4k},\{\pm t,  \pm(k - t), 2k\pm t, 2k\})$, with $k$~odd and $t \equiv k \pmod{4}$, which has Wilson type \pref{Wilson-C2}.

\item \label{val7-Z8}
$\Cay(\ZZ_{8k},\{\pm 4t,\pm k,\pm 3k, 4k\})$, with $k$ and~$t$~odd, which has Wilson type \pref{Wilson-C3}.

\item \label{val7-Z12-12k}
$\Cay(\ZZ_{12k},\{\pm t, \pm(4k - t), \pm(4k + t), 6k\})$, with $t$~odd, which has Wilson type \pref{Wilson-C3}.

\end{enumerate}
\end{thm}

\begin{rem}
It is easy to see that each connection set listed in \cref{val7} contains both even elements and odd elements, so none of the graphs are bipartite. Then it follows from \fullcref{nottwinfree}{prime} that the graphs are also twin-free. Therefore, a graph in the list is nontrivially unstable if and only if it is connected. And this is easy to check: a graph listed in \cref{val7} is connected if and only if $\gcd(t,k) = 1$ (except that the condition for part~\pref{val7-bc-odd} is $\gcd(b, c, k) = 1$). 
\end{rem}

\begin{rem}
In the statement of \cref{val7}, it is implicitly assumed that $k,t \in \ZZ^+$. In order for the graphs to have valency~$7$, the parameters $k$ and~$t$ (or $k$, $b$, and~$c$) must be chosen so that all of the listed elements of the connection set are distinct in the cyclic group. (Note that this implies $k > 1$ in parts \pref{val7-Z4}, \pref{val7-Z6}, and~\pref{val7-Z8}.)
\end{rem}

\begin{proof}[\bf Proof of \cref{val7}]
$(\Leftarrow)$ 
This is the easy direction.
For each family of graphs in the statement of the \lcnamecref{val7}, we briefly justify the specified Wilson type (which implies that the graphs are unstable):
	\begin{enumerate}

	\item[\pref{val7-Z6}] Type~\pref{Wilson-C1} with $S_e=\{\pm2t, \pm2(k - t), \pm2(k + t)\}$ and $2k + S_e = S_e$.

	\item[\pref{val7-bc-odd}] Type~\pref{Wilson-C1} with $S_e=\{\pm 2k,6k\}$ and $4k+S_e=S_e$.

	\item[\pref{val7-Z20}] Type~\pref{Wilson-C1} with $S_e=\{\pm2k,\pm6k,\pm10k\}$ and $4k+S_e=S_e$.

	\item[\pref{val7-Z4}] Type~\pref{Wilson-C2} with $h = k$ and $S_o = \{\pm t, 2k \pm t\}$. Note that $s \equiv 0$ or~$-k \pmod{4}$ if and only if 
		\[ s \in \{ -t, k - t, -(k - t), 2k + t \}, \]
in which case 
		\begin{align*}
		 s + k &\in \{ k - t, 2k - t, t, -k + t\} 
		. \end{align*}
	(Note that $-k + t = -(k - t)$ is in the connection set.)

	\item[\pref{val7-Z8}] Type~\pref{Wilson-C3} with $H = \langle 2k\rangle$, $R=\{\pm 4t, 4k\}$, $d = 4$, $n/d = 2k$, $\{r/d\} = \{t, k\}$, and $H \not\subseteq d\ZZ_{8k}$.

	\item[\pref{val7-Z12-12k}] Type~\pref{Wilson-C3} with $H=\langle 4k\rangle$, $R=\{6k\}$, $d = 6k$, $n/d = 2$, $\{r/d\} = \{1\}$, and $H \not\subseteq d\ZZ_{12k}$.
\end{enumerate}

\bigbreak

($\Rightarrow$) Let $X = \Cay(\ZZ_n, S)$ be a nontrivially unstable circulant graph of valency $7$. Since the graph has odd valency (or by \cref{OddCirculant}), $n$ must be even. We will write $S=\{\pm a, \pm b, \pm c, \nh\}$ for its connection set.

\refstepcounter{caseholder} 

\begin{case} \label{val7pf-NotInvt}
Assume there exists an automorphism of $BX$ that maps an $\nh$-edge to an $a$-edge. 
\end{case}
By \cref{order2orbit}, $S$ must contain every generator of $\langle a \rangle$. It follows that $\phi(|a|) \leq 6$ (and we also know $|a| \neq 2$). In particular, we obtain that 
	\[ \text{$|a|$ is an element of the set $\{3,4,5,6,7,8,9,10,12,14,18\}$.} \]
We consider the following cases.

\begin{subcase}
Assume $|a| \in \{7,9\}$.
\end{subcase}
Then $S$ must contain the $6$ generators of $\langle a \rangle$. Since $|a|$ is odd, we know $\nh \notin \langle a \rangle$, so we conclude that $\ZZ_n = \langle a \rangle \times \langle \nh \rangle$. In particular, $X$ is one of the following: 
	\[ \text{$\Cay(\ZZ_{14},\{\pm 2, \pm 4, \pm 6, 7\})$ or $\Cay(\ZZ_{18},\{\pm 2, \pm 4, \pm 8, 9\})$.} \]
Note that the first graph appears in \fullcref{small}{14-2-4-6-7}, which contradicts the assumption that $X$ is (nontrivially) unstable. The second graph is listed in part~\pref{val7-Z6} of the statement of the \lcnamecref{val7} (with parameters $k=3$ and $t = 1$).

\begin{subcase}
Assume $|a| \in \{14,18\}$.
\end{subcase}
Then $S$ again contains the $6$ generators of $\langle a \rangle$. Additionally, since $|a|$ is even, we have $\nh \in \langle a \rangle$. We conclude that $a$ generates $\ZZ_n$ and it follows that $X$ is one of the following: 
	\[ \text{$\Cay(\ZZ_{14},\{\pm 1, \pm 3, \pm 5, 7\})$ or $\Cay(\ZZ_{18},\{\pm 1, \pm 5, \pm 7, 9\})$.} \]
It is clear that both of these graphs are bipartite, which contradicts the assumption that $X$ is nontrivially unstable.

\begin{subcase} \label{val7Pf-NotInvt-a=5}
Assume $|a| = 5$.
\end{subcase}
Then $S$ contains the $4$ generators of $\langle a \rangle$. We may suppose without loss of generality that besides $\pm a$, the remaining two generators of $\langle a \rangle$ are $\pm b$.

\begin{subsubcase}
Assume $|c| \in \{3,4,6\}$.
\end{subsubcase}
Then $X$ is one of the following graphs, all of which are stable by \cref{small}.
\begin{enumerate}
\item $|c| = 3 \implies X = \Cay(\ZZ_{30},\{\pm 6,\pm 10,\pm 12, 15\})$. See \fullcref{small}{30-6-10-12-15}.
\item $|c| = 4 \implies X = \Cay(\ZZ_{20},\{\pm4,\pm5,\pm8,10\})$. See \fullcref{small}{20-4-5-8-10}.
\item $|c| = 6 \implies X = \Cay(\ZZ_{30},\{\pm5,\pm6,\pm12,15\})$. See \fullcref{small}{30-5-6-12-15}.
\end{enumerate}

\begin{subsubcase} \label{cinv}
Assume $|c| \not \in \{3,4,6\}$.
\end{subsubcase}
From here, $\langle c \rangle$ has more than two generators, while $\langle a\rangle = \langle b\rangle$ has no generators besides $\pm a,\pm b$. Hence, there cannot exists an automorphism of $BX$ mapping an $s$-edge, with $s \in S\setminus \{\pm c\}$, onto a $c$-edge, since \cref{order2orbit} would then imply that $S$ contains all generators of $\langle c \rangle$.
It follows the set of $c$-edges is invariant under $\Aut BX$. Hence, every automorphism of $BX$ maps $S_0$-edges to $S_0$-edges, where we define $S_0\coloneqq \{\pm a, \pm b, \nh\}$. The connected components of $\Cay(\ZZ_n,S_0)$ are isomorphic to 
	\[ \Cay(\ZZ_{10},\{2,4,5,6,8\}) . \]
By \fullcref{small}{10-2-4-5}, this graph is stable. It then follows from \cref{StableSubgraph} that $X$ is stable as well, a contradiction.

\begin{subcase}
Assume $|a| = 8$.
\end{subcase}
We may assume (as in \cref{val7Pf-NotInvt-a=5}) that $\pm a$ and $\pm b$ are the four generators of~$\langle a \rangle$ (which contains $\nh$, because $|a|$ is even).

\begin{subsubcase}
Assume $|c|\in \{3,4,6\}$.
\end{subsubcase}
Then $X$ is one of the following graphs
\begin{enumerate}
\item $|c| = 3$ $\implies$ $X  = \Cay(\ZZ_{24},\{\pm3,\pm8,\pm9,12\})$. This is stable, by \fullcref{small}{24-3-8-9-12}.
\item $|c| = 4$ $\implies$ $X = \Cay(\ZZ_{8},\{\pm 1,\pm2,\pm3,4\})\cong K_8$. This is stable, by \cref{KnStable}.
\item $|c| = 6$ $\implies$ $X = \Cay(\ZZ_{24},\{\pm3,\pm4,\pm9,12\})$. This graph appears in part~\pref{val7-Z8} of the statement of the \lcnamecref{val7} (with parameters $k=3$ and $t = 1$). 
\end{enumerate}

\begin{subsubcase}
Assume $|c|\notin \{3,4,6\}$.
\end{subsubcase}
By the assumption of \cref{val7pf-NotInvt}, there is an automorphism~$\alpha$ of~$X$ that maps an $\nh$-edge to an $a$-edge. By composing with translations on the left and right, we may assume that $\alpha$ fixes the vertex $(0,0)$, and maps an $\nh$-edge that is adjacent to $(0,0)$ to an $a$-edge that is adjacent to $(0,0)$.

We see from \cref{order2orbit} (by the same argument as in \cref{cinv}) that the set of $c$-edges is invariant under all automorphisms of $BX$. So $\alpha$ restricts to an automorphism of $BX_0'$, where $X_0'$ is the connected component of $\Cay(\ZZ_n,\{\pm a, \pm b, \nh\})$ that contains~$0$. However, since $X_0' \cong \Cay(\ZZ_8,\{\pm1,\pm3,4\})$, we know from \fullcref{val5Z8}{invt} that the set of $\nh$-edges is invariant under all automorphisms of $BX_0'$. This contradicts the choice of~$\alpha$. 

\begin{subcase}
Assume $|a| = 10$.
\end{subcase}
Once again, we may assume that $\pm a$ and $\pm b$ are the four generators of~$\langle a \rangle$ (which contains $\nh$, because $|a|$ is even).

\begin{subsubcase}
Assume $|c|\in \{3,4,6\}$.
\end{subsubcase}
Then $X$ is one of the following graphs:
\begin{enumerate}
\item $|c| = 3$ $\implies$ $X = \Cay(\ZZ_{30},\{\pm 3,\pm9,\pm10,15\})$. By \fullcref{small}{30-3-9-10-15}, this graph is stable.
\item $|c| = 4$ $\implies$ $X = \Cay(\ZZ_{20},\{\pm 2,\pm5,\pm 6,10\})$. This graph appears in part~\pref{val7-Z20} of the statement of the \lcnamecref{val7} (with parameters $k=1$ and $t = 5$).
\item $|c| = 6$ $\implies$ $X = \Cay(\ZZ_{30},\{\pm3,\pm5,\pm9,15\})$. This is a bipartite graph, so trivially unstable.
\end{enumerate}

\begin{subsubcase}
Assume $|c|\notin \{3,4,6\}$.
\end{subsubcase}
Since $|a| = 10$, we may write $n = 10m$. 

As before, we see from \cref{order2orbit} that the set of $c$-edges is invariant under all automorphisms of $BX$. Then \cref{StableSubgraph} implies that $|c|$ is even (since cycles of odd length are stable). 

If $m$ is odd, then $c$ must also be odd (since $|c|$ is odd and $n = 10m$ is not divisible by~$4$), so all elements of~$S$ are odd. Then $X$ is bipartite, which contradicts the assumption that $X$ is nontrivially unstable.

So $m$ must be even, which means we may write $m = 2k$. In this notation, we have $X = \Cay(\ZZ_{20k},\{\pm c, \pm 2k, \pm 6k,10k\})$. Note that $c$ must be odd, since $X$ is not bipartite, so this is listed in part~\pref{val7-Z20} of the \lcnamecref{val7} (with parameter $t = c$).

\begin{subcase}
Assume $|a| = 12$. 
\end{subcase}
Then $S$ contains the $4$ generators of $\langle a \rangle$, which are without loss of generality $\pm a$ and $\pm b$.

\begin{subsubcase}
Assume $|c| \in \{3,4,6\}$.
\end{subsubcase}
Then $X$ is one of the following graphs:
\begin{enumerate}
\item $|c| = 3$ $\implies$ $X = \Cay(\ZZ_{12},\{\pm1,\pm4,\pm5,6\})$. This is listed in part~\pref{val7-Z4} of the \lcnamecref{val7} (with parameters $k=3$ and $t = -1$).
\item $|c| = 4$ $\implies$ $X = \Cay(\ZZ_{12},\{\pm1,\pm3,\pm5,6\})$. This is listed in part~\pref{val7-Z12-12k} of the \lcnamecref{val7} (with parameters $k = t = 1$).
\item $|c| = 6$ $\implies$ $X = \Cay(\ZZ_{12},\{\pm1, \pm2, \pm5, 6\})$. This is listed in part~\pref{val7-bc-odd} of the \lcnamecref{val7} (with parameters $k = 1$, $b = 1$, and $c = 5$).
\end{enumerate}

\begin{subsubcase}
Assume $|c| \notin \{3,4,6\}$.
\end{subsubcase}
As usual, we see from \cref{order2orbit} (by the same argument as in \cref{cinv}) that the set of $c$-edges is invariant under all automorphisms of $BX$. Therefore, if we let $S_0\coloneqq\{\pm a, \pm b, \nh\}$, then the set of $S_0$-edges is also invariant under every automorphism of~$BX$. Since the connected components of $\Cay(\ZZ_n,S_0)$ are isomorphic to $\Cay(\ZZ_{12},\{\pm1,\pm5,6\})$, we see from \fullcref{small}{12-1-5-6} that these connected components are stable. It therefore follows from \cref{StableSubgraph} that $X$ is stable.

\begin{subcase}
Assume $|a|\in \{3,4,6\}$.
\end{subcase}

\begin{subsubcase}\label{val7pf-case1-order-ab-346}
Assume $|b|\in \{3,4,6\}$.
\end{subsubcase}
Note that no two of $a$, $b$, and~$c$ can have the same order, as they cannot generate the same subgroup, since a cyclic group of order $3$, $4$, or~$6$ has only $2$ generators. 

If $|c| \in \{3,4,6\}$, then $\{|a|, |b|, |c|\} = \{3,4,6\}$, so 
	\[ X = \Cay(\ZZ_{12},\{\pm2,\pm3,\pm4,6\}). \]
By \fullcref{small}{12-2-3-4-6}, this graph is stable.

So we must have $|c|\notin \{3,4,6\}$. Then, yet again, \cref{order2orbit} implies that the set of $c$-edges is invariant under every automorphism of $BX$. Therefore, if we let $S_0 \coloneqq \{\pm a, \pm b, \nh \}$, then the set of $S_0$-edges is also invariant.

\begin{itemize}
	\item If $\{|a|,|b|\} = \{3,4\}$, then the connected components of $\Cay(\ZZ_n,S_0)$ are isomorphic to $\Cay(\ZZ_{12},\{\pm3,\pm4,6\})$. By \fullcref{small}{12-3-4-6}, this graph is stable and consequently, by \cref{StableSubgraph}, so is $X$.

	\item If $\{|a|,|b|\} = \{3,6\}$, then the connected components of $\Cay(\ZZ_n,S_0)$ are isomorphic to $\Cay(\ZZ_6,\{\pm1,\pm2,3\})$, which, in turn, is isomorphic to $K_6$. By \cref{KnStable}, this is stable. Then, by \cref{StableSubgraph}, so is $X$.
\end{itemize}
Therefore, we must have $\{|a|,|b|\} = \{4,6\}$. 

In this situation, a connected component $X_0'$ of $\Cay(\ZZ_n, S_0)$ is isomorphic to $\Cay(\ZZ_{12},\{\pm 2,\pm 3,6\})$. This graph is nontrivially unstable, as it is listed in \fullcref{val5}{Z12t} with parameters $k=1$ and $s=3$ (and \cref{val5Nontrivial} tells us that it is nontrivially unstable, not merely unstable).
So \cref{val5orbits} tells us that the group $\Aut BX_0'$ has precisely two orbits on the edges of $BX_0'$. We know from the assumption of \cref{val7pf-NotInvt} that there exists an automorphism of $BX$ that maps an $\nh$-edge to an $a$-edge. Since the set of $S_0$-edges is invariant, this implies there is an automorphism of $BX_0'$ that maps an $\nh$-edge to an $a$-edge. (So the $\nh$-edges are in the same orbit as the $a$-edges.) It follows that the set of $b$-edges is invariant. 

If $|a|=4$, then the invariant subgraph $\Cay(\ZZ_n,\{\pm a,\nh\})$ of $X$ has connected components isomorphic to $K_4$, which is stable by \cref{KnStable}. It then follows from \cref{StableSubgraph} that $X$ is stable.

So we must have $|a|=6$. Then $|b|=4$. Write $n=12k$. Then, since $|a|$ is not divisible by~$4$, we see that $a$ must be even. 

If $b$ and $c$ are of opposite parity, then the connected components of $\Cay(\ZZ_n,\{\pm b,\pm c\})$ are nonbipartite. It is not difficult to see that they are also twin-free (since $|b| = 4$ and the valency of the graph is so small).\refnote{val7pf-case1-order-ab-346-aid} Since the set of $b$-edges is invariant under $\Aut BX$ and the set of $c$-edges is also invariant, we know that the set of $\{b,c\}$-edges is invariant. It therefore follows from \fullcref{Val4Subgraph}{even} that $X$ is stable. 

So $b$ and $c$ must have the same parity. However, they cannot both be even, since $a$ is known to be even, and $X$ is not connected if every element of~$S$ is even. So $b$ and~$c$ are odd. Since $|a| = 6$, we now see that $X$ is listed in part~\pref{val7-bc-odd} of the \lcnamecref{val7}.

\begin{subsubcase} \label{val7pf-case1-order-b-not-346}
Assume $|b|\notin \{3,4,6\}$.
\end{subsubcase}
By symmetry, we may additionally suppose that $|c|\notin \{3,4,6\}$. Then, by applying \cref{order2orbit} one last time, we see that no automorphism of $BX$ maps an $\nh$-edge to a $b$-edge or a $c$-edge. This implies that the set of $\{a, \nh\}$-edges is invariant under all automorphisms of $BX$, and the set of $\{b,c\}$-edges is also invariant.

If $|a| = 3$, then the connected components of $\Cay(\ZZ_n,S_0)$ are isomorphic to $\Cay(\ZZ_6,\{\pm 2,3\})$, which is stable by \fullcref{small}{6-2-3}. It follows by \cref{StableSubgraph} that $X$ is also stable. 

If $|a| = 4$, then the connected components of $\Cay(\ZZ_n,S_0)$ are isomorphic to $\Cay(\ZZ_4,\{1,2,3\})$, which is further isomorphic to $K_4$, which is stable by \cref{KnStable}. It follows again by \cref{StableSubgraph} that $X$ is stable.

We now only have the case $|a| = 6$ to consider. Let $n = 6m$. Note that then $\{\pm a, \nh\} = \{\pm m,3m\}$. We focus on 
	\[ X_0 \coloneqq \Cay(\ZZ_n, \{\pm b, \pm c\}) .\]
Let $X_0'$ be the connected component of $X_0$ that contains~$0$. As $X$ is assumed to be unstable and $X_0$ is $4$-valent, we can conclude from \cref{Val4Subgraph} that $|V(X_0')|$ is even and that either $X_0'$ is bipartite or $X_0'$ is not twin-free. 

Suppose, first, that $X_0'$ is bipartite. This implies that $b$ and $c$ are of the same parity.
	\begin{enumerate}
	\item If $b$ and $c$ are even, then $m$ must be odd. (Otherwise, $S$ would contain only even integers, so $X$ would not be connected.) Consequently, since $n = 6m$, it follows that $b$ and $c$ are both of odd order, so $X_0'$ contains an odd cycle. This contradicts the assumption that $X_0'$ is bipartite.

	\item If $b$ and $c$ are odd, then $m$ must be even. (Otherwise, every element of~$S$ is odd, so $X$ is bipartite, which contradicts the fact that $X$ is \emph{nontrivially} unstable.). Hence, we may write $m = 2k$, and then we see that $X$ is listed in part~\pref{val7-bc-odd} of the \lcnamecref{val7}. 
	\end{enumerate}

We can now assume that $X_0'$ is not bipartite and is not twin-free. By \fullcref{nottwinfree}{d=4}, it follows that $X_0'$ is isomorphic to $K_{4,4}$ or $C_\ell\wr \overline{K_2}$ with $\ell=|V(X_0')| / 2$. As $X_0'$ is assumed to be nonbipartite, the first case is not possible, and in the second case, we see that $|V(X_0')|/2$ is odd.

From the fact that $X_0'$ is not twin-free (and the valency of~$X$ is small --- only~$4$\refnote{val7pf-case1-order-b-not-346-AID}) we obtain that $c = b + \nh$ (perhaps after replacing $b$ with~$-b$). For every $v \in \ZZ_n \times \ZZ_2$, we deduce that $v + (\nh,0)$ is the unique twin of~$v$ in the graph $BX_0$. Since automorphisms must map twin vertices to twin vertices, and the set of $\{b,c\}$-edges is invariant under $\Aut BX$, we conclude that the cosets of the subgroup $\langle (\nh,0)\rangle$ are blocks for the action of $\Aut BX$ (see \cref{BlockDefn}). Note that quotient graph of $BX_0'$ with respect to the partition induced by cosets of $\langle (\nh,0)\rangle$ is a cycle. Since $|V(X_0')|/2$ is odd, the length of this cycle is $|V(X_0')|$. Therefore, in this cycle, the vertices corresponding to the cosets $\{(0,0),(\nh,0)\}$ and $\{(0,1),(\nh,1)\}$ are at maximum distance. 

It follows that if $\alpha$ is an automorphism of $BX$ that fixes $(0,0)$, then $\alpha$ must fix the coset $\{(0,1),(\nh,1)\}$ set-wise. Also note that $\alpha$ must fix the set of neighbors of $(0,0)$ in $BX_1$ with $X_1\coloneqq\Cay(\ZZ_n, \{ \pm a, \nh \})$ (because  $\{a, \nh\}$-edges are invariant); this set of neighbors is $\{\pm(a,1),(\nh,1)\}$. Then $\alpha$ must fix the intersection of these two sets, which is $\{(\nh,1)\}$. The automorphism~$\alpha$ must therefore also fix the twin of the vertex $(\nh,1)$, which is $(0,1)$. We now conclude from \cref{StableIffStabilizer} that $X$ is stable.

\medbreak

This completes the proof of \cref{val7pf-NotInvt}. For the remaining cases, we may assume that every automorphism of $BX$ maps $\nh$-edges to $\nh$-edges, and is therefore an automorphism of the canonical double cover of the following subgraph~$X_0$ of~$X$: 

\begin{notation}
For the remainder of the proof, we let 
	\[ X_0 \coloneqq \Cay(\ZZ_n,\{\pm a,\pm b, \pm c\}) \]
be the graph that is obtained from $X$ by removing all of the $\nh$-edges.
\end{notation}

\begin{case}
Assume $X_0$ is not connected. 
\end{case}
(We also assume that every automorphism of $BX$ maps $\nh$-edges to $\nh$-edges, for otherwise \cref{val7pf-NotInvt} applies.)
Then $\langle a,b,c \rangle$ is a proper subgroup of $\ZZ_n$, but $\langle a,b,c,\nh \rangle$ is the whole group $\ZZ_n$. It follows that $n = 2m$, where $m$ is odd, and $|\langle a,b,c \rangle| = m$. Let $X_0'$ be the connected component of $X_0$ that contains~$0$. Note that $X_0'$ is connected by definition and also that it is not bipartite (since it is vertex-transitive and of odd order). 

If $X_0'$ is twin-free, then it follows by \cref{OddCirculant} that $X_0'$ is stable. By \cref{StableSubgraph}, we conclude that $X$ is stable, which is a contradiction.

Therefore, we know that $X_0'$ is not twin-free. Then by \fullcref{nottwinfree}{Y}, $X_0' \cong Y\wr \overline{K_d}$, where $Y$ is a vertex-transitive, connected graph and $d \geq 2$. Let $\delta$ be the valency of~$Y$. Since $X_0'$ is $6$-valent, it follows that $6 = \delta d$ and therefore, $d \in \{2,3,6\}$. Because $d \, |V(Y)| = |V(X_0')|$ is odd, it cannot happen that $d$ is even. Hence, we conclude that $d = 3$. It follows that $\delta = 2$, so $Y$ must be a cycle. Letting 
	\[ k := |V(Y)| = m/3 = n/6 , \] 
we conclude that $X_0' \cong C_k \wr \overline{K_3}$. Since $X_0$ is the disjoint union of two copies of~$X_0'$, we now see that
	\[ X_0 = \Cay(\ZZ_{6k}, \{ \pm s, 2k \pm s, 4k \pm s \}), \]
for some $s \in \ZZ_n$, such that $|\langle s, 2k \rangle| = n/2$. (This final condition means $\gcd(s, 2k, 6k) = 2$, so $s$ must be even; write $s = 2t$. Then
	\[ X = \Cay(\ZZ_{6k}, \{ \pm 2t, 2k \pm 2t, 4k \pm 2t, 3k \}). \]
This graph is listed in part~\pref{val7-Z6} of the \lcnamecref{val7}.

\begin{case}
Assume that $X_0$ is bipartite and that $BX_0$ is arc-transitive.
\end{case}
(We also assume that every automorphism of $BX$ maps $\nh$-edges to $\nh$-edges, and that $X_0$ is connected, for otherwise a previous case applies.)
Since $X_0$ is bipartite, it follows that $BX_0$ is isomorphic to a disjoint union of two copies of $X_0$. Since $BX_0$ is assumed to be arc-transitive, it follows that $X_0$ is a connected arc-transitive circulant graph. Consequently, it is one of the four types that are listed in \cref{ArcTrans}. Type~\fullref{ArcTrans}{Kn} is impossible because $X_0$ is bipartite (and has valency~$6$). By \cref{ReduceNonNormal}, it follows that $X_0$ is not a normal Cayley graph, so it does not have type~\fullref{ArcTrans}{normal} either.

\begin{subcase}
Assume $X_0$ has type~\fullref{ArcTrans}{wreath}.
\end{subcase}
Then $X_0 = Y\wr \overline{K_d}$, where $Y$ is a connected arc-transitive circulant graph, and $d \ge 2$. Let $\delta$ be the valency of $Y$. Since $X_0$ has valency $6$, it follows that $\delta d=6$, so $d \in \{2, 3, 6\}$.

\begin{subsubcase} \label{val7pt-d=6}
Assume $d=6$.
\end{subsubcase}
Then $X_0\cong K_{6,6}$. It follows that $n=12$ and we obtain $X=\Cay(\ZZ_{12},\{\pm 1, \pm 3, \pm 5, 6\})$. We have already seen this graph; it is listed in part~\pref{val7-Z12-12k} of the \lcnamecref{val7} with parameters $k=t=1$.

\begin{subsubcase}
Assume $d=3$.
\end{subsubcase}
Then $Y$ is a cycle of even length $2m$. We obtain that 
	\[ X_0 = \Cay(\ZZ_{6m},\{\pm t, 2m\pm t, 4m\pm t\}) . \]
Then $X = \Cay(\ZZ_{6m},\{\pm t, 2m\pm t, 4m\pm t, 3m\})$. Since $X_0$ is connected, we have $\gcd(t, 2m) = 1$. In particular, $t$ is odd. Since $X$ is not bipartite, this implies $m$ is even. Writing $m=2k$, we get that 
	\[ X = \Cay(\ZZ_{12k},\{\pm t, 4k\pm t, 8k\pm t,6k\}). \]
So $X$ is listed in part~\pref{val7-Z12-12k} of the \lcnamecref{val7}. 

\begin{subsubcase}
Assume $d=2$.
\end{subsubcase}
Then $Y$ is a connected, cubic, arc-transitive, circulant graph, so \cref{CubicArcTransCirculant} tells us that $Y$ is either $K_4$ or $K_{3,3}$. Since $X_0$ is bipartite, it follows that $Y$ is isomorphic to $K_{3,3}$. We obtain that $X_0 \cong K_{6,6}$, a case that has already been considered in \cref{val7pt-d=6}.

\begin{subcase}
Assume $X_0$ has type~\fullref{ArcTrans}{AlmostWreath}.
\end{subcase}
Letting $n = |V(X_0)|$, we have $X_0 = Y\wr \overline{K_d} - d Y$, where $n = md$, $d > 3$, $\gcd(d,m) = 1$ and $Y$ is a connected arc-transitive circulant graph of order~$m$. Let $\delta$ be the valency of~$Y$. Since $X_0$ has valency~$6$, we must have $\delta (d-1) = 6$. Since $d > 3$, this implies that either $d = 7$ and $\delta = 1$ or $d = 4$ and $\delta = 2$.

If $\delta = 1$, then $Y = K_2$ (so $m = 2$). This implies $n = md = 2 \cdot 7 = 14$. Then $X_0=K_2\wr \overline{K_7}-7K_2$ and $X=K_2\wr\overline{K_7}\cong K_{7,7}$. Hence, $X$ is bipartite and trivially unstable.

Assume, now, that $\delta = 2$ and $d = 4$. Since $\delta = 2$, we have $Y = C_m$. Then $m$~must be even, because $X_0$ is bipartite. This contradicts the fact that $\gcd(d,m) = 1$.

\begin{case} \label{final}
Assume that none of the preceding cases apply.
\end{case}
This means that:
	\begin{enumerate}
	\item every automorphism of $BX$ maps $\nh$-edges to $\nh$-edges,
	\item $X_0$~is connected,
	and
	\item either $X_0$ is not bipartite or $BX_0$ is not arc-transitive.
	\end{enumerate}

\begin{subcase} \label{allinv}
Assume there exists $s \in \{a,b,c\}$, such that the set of $s$-edges is invariant under the action of $\Aut(BX)$, and the graph $\Cay(\ZZ_n, \{\pm s, \nh\})$ is not bipartite.
\end{subcase}
Let $X_1\coloneqq\Cay(\ZZ_n,\{\pm s, \nh\})$, and let $X_1'$ be a connected component of~$X_1$. Then $X_1'$ is connected, cubic, and nonbipartite. The graph $X_1'$ must also be twin-free, because otherwise \fullcref{nottwinfree}{prime} would imply that $X_1'\cong K_{3,3}$, which contradicts the fact that $X_1'$ is not bipartite. We therefore conclude from \cref{val3} that $X_1'$ is stable. By our assumptions, the set of $s$-edges is invariant under the action of $\Aut(BX)$, and the set of $\nh$-edges is also invariant. So it follows by \cref{StableSubgraph} that $X$ is stable, which is a contradiction.

\begin{subcase}
Assume there exists $s \in \{a,b,c\}$, such that the set of $s$-edges is invariant under $\Aut BX$.
\end{subcase}
Since $X$ is not bipartite, we know that $S$ contains two elements of opposite parity. Therefore, we may assume without loss of generality that $a + \nh$ is odd. We may assume that the set of $a$-edges is not invariant under $\Aut BX$, for otherwise \cref{allinv} applies (with $s = a$). So $s \neq a$. Hence, we may assume, without loss of generality, that $s = c$, which means the set of $c$-edges is invariant under $\Aut BX$. This implies that the $a$-edges and the $b$-edges are in the same orbit of $\Aut BX$.

We consider the following two subgraphs of $X$:
	\[ \text{$X_1\coloneqq\Cay(\ZZ_n,\{\pm c, \nh\})$ and $X_2\coloneqq\Cay(\ZZ_n,\{\pm a, \pm b\})$.} \]
Denote their connected components containing~$0$ by $X_1'$ and~$X_2'$, respectively. 

We may assume that $X_1'$ is bipartite. (Otherwise, \cref{allinv} applies with $s = c$.) This implies that $|c|$ is even, so $\nh \in \langle c \rangle$. More precisely, we have $\nh = (|c|/2) \, c$. Since $X_1'$ is bipartite, this implies that $|c|/2$ is odd.

Also, since $X$ is unstable, it follows from \cref{Val4Subgraph} that $|V(X_2')|$ is even and either $X_2'$ is not twin-free or $X_2'$ is bipartite. 

\begin{subsubcase} \label{val7pf-twins}
Assume $X_2'$ is not twin-free. 
\end{subsubcase}
We claim that $X_2'\cong K_{4,4}$. From \fullcref{nottwinfree}{d=4}, we obtain that $X_2'$ is isomorphic to $K_{4,4}$ or $C_{\ell} \wr \overline{K_2}$ with $\ell\coloneqq |\langle a,b\rangle| / 2$. Let $X_2^*$ be the connected component of $\Cay(\ZZ_n,\{\pm a,\pm b,\nh\})$ containing $0$. Note that $X_2^*$ is obtained by adding $\nh$-edges to $X_2'$. Since $X_2'\cong C_\ell\wr\overline{K_2}$, this implies $X_2^*\cong C_{\ell}\wr K_2$. However, since $\nh$-edges are also invariant, $X_2^*$ cannot be stable, since in this case \cref{StableSubgraph} would imply that $X$ is stable. By \cref{CycWr}, we therefore conclude that $\ell = 4$. Finally, note that $C_4\wr \overline{K_2}\cong K_{4,4}$. Hence, we may assume that $X_2'\cong K_{4,4}$. This completes the proof of the claim.

It follows from the claim that $|V(X_2')| = |\langle a, b\rangle|= 8$, so we may write $n=8k$. The claim then implies that $\{\pm a,\pm b\} = \{\pm k,\pm 3k\}$. We also obtain:
	\[8k = n = |\langle a,b,c,\nh\rangle| = |\langle a,b,c\rangle|=\frac{|\langle a,b\rangle| \cdot |c|}{|\langle a,b \rangle \cap \langle c \rangle|} = \frac{8|c|}{|\langle a,b \rangle \cap \langle c \rangle|}\]
This immediately implies that $k$ divides $|c|$. More precisely, since $|c|/2$ is odd, and the denominator of the right-most term is a divisor of $|\langle a,b \rangle| = 8$ (and is a multiple of $|\langle \nh \rangle| = 2$), the only possibility is that $k = |c|/2$ (so $k$ is odd). We conclude that:
\[X = \Cay(\ZZ_{8k},\{\pm c, \pm k, \pm 3k, 4k\}) .\]
Since $|c| = 2k = 8k/4$, we know that $\gcd(c, 8k) = 4$; this means that we may write $c = 4t$ with $t$~odd. So $X$ is listed in part~\pref{val7-Z8} of the \lcnamecref{val7} (with parameter $4t = c$).

\begin{subsubcase}
Assume $X_2'$ is bipartite. 
\end{subsubcase}
Then $a$ and $b$ must be of the same parity, but as $X$ is nonbipartite, it follows that their parity is opposite to that of $c$ and~$\nh$. Therefore, if we let $X_2^*$ be a connected component of $\Cay(\ZZ_n, \{\pm a, \pm b, \nh\})$, then $X_2^*$ is a connected, nonbipartite, $5$-valent circulant graph. 
By \fullcref{nottwinfree}{prime}, it is also twin-free (since it is not bipartite, and therefore cannot be isomorphic to $K_{5,5}$).
We conclude that $X_2^*$ is not trivially unstable. 

Due to \cref{StableSubgraph} and the fact that $X$ is unstable, we see that $X_2^*$ is not stable. Hence, it is nontrivially unstable. From our assumptions, we already know that $a$-edges and $b$-edges are in the same orbit under the action of $\Aut B X_2^*$. This implies that if the set of $\nh$-edges is not invariant, then all edges of $BX_2^*$ are in the same orbit, which would contradict \cref{val5orbits}. Hence, the set of $\nh$-edges must be invariant, so \cref{val5nhInvt} tells us that $X_2^*$ is isomorphic to $\Cay(\ZZ_8,\{ \pm 1, \pm 3,4\})$. Therefore, $X_2' \cong \Cay(\ZZ_8,\{ \pm 1, \pm 3\}) \cong K_{4,4}$ is not twin-free, so \cref{val7pf-twins} applies.

\begin{subcase}
Assume all edges of $BX$ besides $\nh$-edges are in the same orbit of $\Aut BX$.
\end{subcase}
As every automorphism of $BX$ is also an automorphism of $BX_0$, it follows that $BX_0$ is arc-transitive. Then by the assumption of \cref{final}, $X_0$ must be nonbipartite.

\begin{subsubcase} \label{val7pf-a=4}
Assume $|a|=4$.
\end{subsubcase}
Since $BX_0$ is arc-transitive, there exist automorphisms of $BX_0$ mapping an $a$-edge to a $b$-edge and a $c$-edge. Since $|a|=4$, by \cref{order2orbit} it follows that all generators of the subgroups $\langle b \rangle$ and $\langle c \rangle$ are in $S$. As all elements of $S$ are pairwise distinct, it is clear that either $b$ and $c$ generate distinct subgroups with exactly $2$ generators each or $b$ and $c$ generate the same subgroup with exactly $4$ generators.

In the first case, we get that $|b|\neq|c|$ and both lie in $\{3,4,6\}$. As $|a|=4$, we conclude that $\{|a|, |b|, |c| \} =  \{3,4,6\}$, so $X = \Cay(\ZZ_{12},\{\pm2,\pm3,\pm4,6\})$. By part \pref{small-12-2-3-4-6} of \cref{small}, this graph is stable.

In the second case, it follows that $|b|=|c|\in \{5,8,10,12\}$. 
Then $X$ is one of the following graphs:
\begin{itemize}
\item $|b| = |c| = 5$ $\implies$ $X = \Cay(\ZZ_{20},\{\pm 4,\pm 5,\pm 8,10\})$. This is stable by \fullcref{small}{20-4-5-8-10}.
\item $|b| = |c| = 8$ $\implies$ $X = \Cay(\ZZ_8,\{\pm1,\pm2,\pm3,4\}) \cong K_8$. This is stable by \cref{KnStable}.
\item $|b| = |c| = 10$ $\implies$ $X = \Cay(\ZZ_{20},\{\pm2,\pm5,\pm6,10\})$. We find this graph in part~\pref{val7-Z20} of the \lcnamecref{val7} (with parameters $k=1$ and $t = 5$).
\item $|b| = |c| = 12$ $\implies$ $X = \Cay(\ZZ_{12},\{\pm1,\pm3,\pm5,6\}$. We find this graph in part~\pref{val7-Z12-12k} of the \lcnamecref{val7} (with parameters $k=t=1$).
\end{itemize}

\begin{subsubcase}
Assume $2s \neq 2t$, for all $s,t \in S$, such that $s \neq t$.
\end{subsubcase}
Then from \cref{2S'-BX}, it follows that any automorphism of $BX_0$ (and consequently of $BX$, as well) is an automorphism of 
	\[ \Cay(\ZZ_n \times \ZZ_2, \{(\pm 2a,0),(\pm 2b,0),(\pm 2c,0)\}) . \]
Therefore, (the vertex sets of) the connected components of this graph are blocks for the action of $\Aut BX_0$. These blocks are the four cosets of the subgroup $2\ZZ_n \times \{0\}$. As $BX_0$ is $6$-valent and arc-transitive (and all neighbors of $(0,0)$ are in $\ZZ_n \times \{1\}$), it follows that either all $6$ neighbors of $(0,0)$ in $BX_0$ lie in the same coset of $2\ZZ_n\times \{0\}$ or three of its neighbors lie in $2\ZZ_n \times \{1\}$ and the other three lie in $(1+ 2\ZZ_n) \times \{1\}$. In the first case, it follows that $a,b,c$ are all of the same parity, which contradicts the fact that $X_0$ is connected and nonbipartite. In the second case, it follows that exactly three elements of the set $\{\pm a, \pm b,\pm c\}$ are odd, which is impossible, since $-s$ has the same parity as~$s$.

\begin{subsubcase}\label{val7pf-NO-order4-NO-2a=2b}
Assume that neither of the two preceding \lcnamecref{val7pf-a=4}s apply.
\end{subsubcase}
This means that:
	\begin{enumerate}
	\item  $S$ contains no element of order $4$,
	 and
	 \item we may assume, without loss of generality, that $2a = 2b$.
	 \end{enumerate}
By \cref{2S'-BX}, every automorphism of $BX_0$ is also an automorphism of $\Cay(\ZZ_n \times \ZZ_2,\{(\pm 2c,0)\})$. This implies that the cosets of $\langle 2c \rangle \times \{0\}$ are blocks for the action of $\Aut BX_0$. The two $c$-neighbors of $(0,0)$ are both in the coset $(c,1) + \bigl( \langle 2c \rangle \times \{0\} \bigr)$. Therefore, by arc-transitivity, either all neighbors of $(0,0)$ in $BX_0$ are in this coset or there are three different cosets, each containing two neighbors of $(0,0)$. 

However, if all neighbors of $(0,0)$ are in $(c,1) + \bigl( \langle 2c \rangle \times \{0\} \bigr)$, then $a$ and $b$ have the same parity as $c$. This contradicts the fact that $X_0$ is connected and nonbipartite. 

So there are three different cosets that each contain two neighbors of $(0,0)$. Consider the quotient graph of $BX_0$ with respect to the coset partition induced by $\langle 2c \rangle \times \{0\}$. This is a cubic, connected, bipartite, arc-transitive graph $Q$, which is a Cayley graph on $\ZZ_m \times \ZZ_2$, where $m$ is the index of $\langle 2c \rangle$ in $\ZZ_n$. It follows from \cref{ArcTrans} that there are only three cubic, connected Cayley graphs on abelian groups (up to isomorphism): $K_4$, $K_{3,3}$, and the cube~$Q_3$. 
	\begin{itemize}
	\item $K_4$ is not bipartite. 
	\item If $Q\cong K_{3,3}$, then $\langle 2c \rangle$ is of index $3$ in $\ZZ_n$. But this means $\langle 2c \rangle$ cannot be of index $2$ in $\langle c \rangle$, so it follows that $\langle 2c \rangle = \langle c \rangle$, so $c$ is of odd order. Then $n=3|c|$ is also odd. By \cref{OddCirculant}, $X$ is stable. 
	\end{itemize}
Therefore, $Q$ must be the cube (and therefore has exactly 8 vertices).

It follows that $\langle 2c \rangle \times 0$ is of index $8$ in $\ZZ_n \times \ZZ_2$, so the order of $2c$ is $n/4$. This means $n / {\gcd(n,2c)} = n/4$, so $\gcd(n,2c)=4$. Therefore, $c$ is an even integer, and $\nh$ is also even. Since $X$ is connected, then $a$ and $b=a+\nh$ must be odd. 

Note that the five non-zero neighbors of $(a,1)$ in $BX_0$ are 
	\[ (a \pm c, 0),\ (2a, 0) = (2b, 0), \ (a \pm b, 0) .\]
There is more than one path of length~$2$ from $(0,0)$ to each of these vertices. Due to arc-transitivity, this implies that the the path $(0,0), (c,1), (2c,0)$ is not the only path of length~$2$ from $(0,0)$ to $(2c,0)$. Because the first coordinate of $(2c,0)$ is an even integer, we conclude that $2c$ is a sum of two integers from $\{\pm a, \pm a+ \nh,\pm c\}$ of the same parity, besides $c+c$. As we have assumed there is no element of order $4$ in $S$, the case $2c = (-c)+(-c)$ is not possible. So $2c$ must be a sum of two (odd) integers from $\{\pm a, \pm a+ \nh\}$. Recalling that $2c \neq 2s$ for $s \neq c$, we see that (up to relabeling $a,-a$ and $a+\nh$), we must have either 
	\[ \text{$2c = a + (a + \nh) = 2a + \nh$ \ or \ $2c = a - (a + \nh) = \nh$.} \]
However, the case $2c = \nh$ is not possible, because no element of~$S$ has order~$4$. Therefore, we have $2c = 2a + \nh$. This means $2(c - a) = \nh$, so either $c = a + (n/4)$ or $c = a - (n/4)$. However, if $c = a + (n/4)$, then $-c = -a - (n/4)$; therefore, we may assume $c = a - (n/4)$, by replacing $a$, $b$, and~$c$ with their negatives, if necessary.

As $a$ is an odd integer, we know that $n/|a| = \gcd(a, n)$ is odd, so $\langle a \rangle$ contains every element of~$\ZZ_n$ whose order is a power of~$2$. In particular, it contains $\nh$ and $\pm n/4$. Since $b = a + \nh$ and $c = a + (n/4)$, we conclude that $\langle a \rangle = \ZZ_n$. Then, writing $n = 4k$, we have:
	\[X = \Cay(\ZZ_{4k},\{\pm a, \pm(a + 2k), \pm(a - k), 2k\}) . \]
If $a \equiv k \pmod{4}$, then this is listed in part~\pref{val7-Z4} of the \lcnamecref{val7} (with parameter $t = a$). 

To complete the proof, we will show that if $a \not\equiv k \pmod{4}$, then the subgraph~$X_0$ is stable. This then implies by \cref{StableSubgraph} that $X$ is stable as well, which is a contradiction.

Suppose, for a contradiction, that $X_0$ is unstable (and $a \equiv -k \pmod{4}$). To work around a conflict of notation, let us change our notation for~$X_0$, by writing $\alpha$ and~$\kappa$ instead of $a$ and~$k$:
		\[X_0 = \Cay(\ZZ_{4 \kappa},\{\pm \alpha, \pm(\alpha + 2\kappa), \pm(\alpha - \kappa)\}) .\]
Recall that $\kappa$ is odd, and $\alpha \equiv -\kappa \pmod{4}$.

Since $X_0$ is connected, nonbipartite, and twin-free, \refnote{val7pf-NO-order4-NO-2a=2b-AID}it is nontrivially unstable. So it must appear in the list of nontrivially unstable $6$-valent graphs in \cref{explicit6}. However:
	\begin{enumerate}
	\item Since $\kappa$ is odd, we know that $|V(X_0)|=4\kappa$ is not divisible by $8$. So $X_0$ cannot appear under \pref{explicit6-C1Se2}, \pref{explicit6-C3order2}, or \pref{explicit6-C3order4} in \cref{explicit6}. 
	\item Since the connection set of~$X_0$ has $4$ odd elements ($\pm \alpha$ and $\pm(\alpha + 2\kappa)$), we can also rule out the family \fullref{explicit6}{C1Se4}.
	\item If $X_0$ is a member of the family~\fullref{explicit6}{C2}, then, since $\alpha - k$ and its negative are the only even elements of the connection set for~$X_0$ (and the four elements $\pm(a + k)$ and $\pm(a - k)$ all have the same parity), the element~$a$ of~\fullref{explicit6}{C2} must be $\alpha - \kappa$ (perhaps after replacing $a$ with~$-a$). Since $\alpha \equiv -\kappa \pmod{4}$, this implies 
		\[ a = \alpha - \kappa \equiv 2\alpha \equiv 2 \pmod{4} , \]
	which contradicts the requirement of~\fullref{explicit6}{C2} that $a \equiv 0 \pmod{4}$. 
	\item Assume $X_0$ is a member of the family~\fullref{explicit6}{C4more}. Then we can find $m \in\ZZ$ with $\gcd(m,4\kappa)=1$ and $m^2\equiv 1\pmod{4\kappa}$ satisfying the identity listed in \fullref{explicit6}{C4more}. Using the notation from \cref{explicit6}, we observe the following cases:
	\begin{itemize}
	\item \emph{Assume $a=\alpha$.} We then obtain that:	
	\[(m-1)\alpha \equiv (m+1)(\alpha+2\kappa)\pmod{4\kappa},\] so $2\mid(\alpha+\kappa(m+1))$. This is clearly a contradiction, since the right-most expression is odd, since $\alpha$ and $m$ are odd.
	\item \emph{Assume $a=\alpha+2\kappa$.} We then obtain a contradiction of the same type:
	\[(m-1)(\alpha+2\kappa) \equiv (m+1)\alpha \pmod{4\kappa},\] so $2\mid (\kappa(m-1)-\alpha)$.
	\item \emph{Assume $a=\alpha-\kappa$.} It then follows that:
	\[(m-1)(\alpha-\kappa)\equiv (m+1)\alpha\pmod{4\kappa},\] so $\kappa\mid \alpha$.
	This is a contradiction, since we have already established that $\alpha$ is a generator of the group $\ZZ_{4\kappa}$.
	\end{itemize}
	\item Finally, suppose $X_0$ is a member of the family~\fullref{explicit6}{C4}.	
Since $\alpha - \kappa$ and its negative are the only even elements of the connection set of~$X_0$ (and the four elements $\pm b$ and $\pm m b + 2k$ all have the same parity), the element~$a$ of~\fullref{explicit6}{C2} must be $\alpha - \kappa$ (or its negative). Then $a$ and $m - 1$ are even (and $k = n/4 = \kappa$ is odd), so we have
	\[ (m - 1) a \equiv 0 \not\equiv 2 \equiv 2k \pmod{4} . \]
This contradicts the requirement that $(m - 1)a \equiv 2k \pmod{n}$.
	\end{enumerate}
Therefore $X_0$ does not appear on any of the lists in \cref{explicit6}. This contradiction completes this final case of the proof of the \lcnamecref{val7}.
\end{proof}

\AtEndDocument{\input{LowValencyNotesToReferee}}  

\end{document}

%% file: LowValencyNotesToReferee.tex


\newpage
\addtocontents{toc}{\vskip\bigskipamount}
\begin{appendix}

\section{Notes to aid the referee}


\begin{aid} \label{BlockDefn-H}
Let $G^*$ be the regular representation of $G$. Since $G^*$ is contained in $\Aut X$, it is clear that $\mathcal{B}$ is also a block for the action of~$G^*$. Hence, \cite[Exer.~1.5.6, p.~13]{DixonMortimer} tells us that $\mathcal{B}$ is an orbit of some subgroup~$H^*$ of~$G^*$. If we let $H$ be the corresponding subgroup of $G$, then the $H^*$ orbit of an element~$b$ of $G$ is precisely $b + H$, so it is a coset of~$H$.

We see from \cite[Exer.~1.5.3, p.~13]{DixonMortimer} that every coset of~$H$ is a block. Furthermore, since these cosets form a partition of $G$, we see that they are the entire ``system of blocks'' containing~$\mathcal{B}$, so \cite[Exer.~1.5.3, p.~12]{DixonMortimer} tells us that $\Aut X$ acts on this set of blocks.
\end{aid}


\begin{aid}\label{val6pf-same-coset}
Suppose $u$ and~$v$ are two distinct vertices in the same coset of $\langle (k, 0) \rangle$. Since $|k| = 4$, we may assume that $v$ is either $u + (k, 0)$ or $u + (2k, 0)$ (after interchanging $u$ and~$v$ if $v = u + (3k, 0)$).

If $v = u + (k, 0)$, then both vertices are adjacent to
        \begin{align*}
        u + (c + k, 0) &= v + (c, 0) \\
        u + (c, 0) &= v + (c - k, 0) \\
        u - (c, 0) &= v - (c + k, 0) \\
        u - (c - k, 0) &= v - (c, 0)
        \end{align*}
If $v = u + (2k, 0)$, then both vertices are adjacent to
        \begin{align*}
        u + (c + k, 0) &= v + (c - k, 0) \\
        u + (c - k, 0) &= v + (c + k, 0) \\
        u - (c + k, 0) &= v - (c - k, 0) \\
        u - (c - k, 0) &= v - (c + k, 0)
        \end{align*}
\end{aid}

\begin{aid}\label{val6pf-diff-cosets}
Assume $u$ and~$v$ are vertices of~$BX$ that are in different cosets of $\langle (k, 0) \rangle$, and let $w$ be a common neighbor. Since $S \subseteq (c + \langle k \rangle) \cup (-c + \langle k \rangle)$, we may assume $w \in u + \tilde c + \langle (k, 0) \rangle$ and $w \in v - \tilde c + \langle (k, 0) \rangle$, by interchanging $u$ and~$v$ if necessary.

Since $c \notin \langle k \rangle$ and $|\ZZ_{4k}/\langle k \rangle| = k = n/4$ is odd, we know that $4c \notin \langle k \rangle$. Therefore, taking congruences modulo $\langle (k, 0) \rangle$, we have
        \begin{align*}
        u - \tilde c
        &= u + \tilde c - (2c, 0)
        \\& \equiv w - (2c, 0)
        \\&\equiv v - \tilde c - (2c, 0)
        \\&= v + \tilde c - (4c, 0)
        \\&\not\equiv  v + \tilde c
                && \begin{pmatrix} \text{$4c \notin \langle k \rangle$, so} \\ (4c, 0) \not\equiv 0 \pmod{} \end{pmatrix}
        . \end{align*}
This implies that no vertex in $u - c + \langle (k, 0) \rangle$ is adjacent to~$v$, so all of the common neighbors are in $u + c + \langle (k, 0) \rangle$. Since the only neighbors of $u$ that are in this coset are $u + c$, $u + (c + k)$, and $u + (c - k)$, this obviously implies that $u$ and~$v$ have no more than~3 common neighbors.
\end{aid}

\begin{aid} \label{val7pf-case1-order-ab-346-aid}
Details missing in the proof of \cref{val7pf-case1-order-ab-346}
\begin{proof}
We are assuming that $b$ and $c$ are of opposite parity, so the graph $\Gamma\coloneqq\Cay(\ZZ_n,\{\pm b,\pm c\})$ is nonbiparite. Assume that it is not twin-free. Then by \fullcref{nottwinfree}{d=4} and the fact that $\Gamma$ is nonbipartite, it follows that
$\Gamma \cong C_\ell \wr \overline{K_2}$ with $\ell = |V(\Gamma)|/2$. Furthermore, the unique twin of $0$ is $\nh$ i.e. $\nh + \{\pm b,\pm c\} = \{\pm b,\pm c\}$. Since $|b|=4$, it follows that $b+\nh = -b$ and $-b+\nh = b$. Hence, $c + \nh = -c$ and $-c + \nh = c$, which implies $|c|$ is a divisor of $4$, a contradiction.
\end{proof}
\end{aid}

\begin{aid} \label{val7pf-case1-order-b-not-346-AID}
Details missing in the proof of \cref{val7pf-case1-order-b-not-346}.
\begin{proof}
By \fullcref{nottwinfree}{d=4}, $0$ and $\nh$ are twins. It follows that $\nh + \{\pm b,\pm c\} = \{\pm b,\pm c\}$. Clearly, $b+\nh \neq b$ and since $|b|\not\in \{2,4\}$, it follows that $b+\nh \neq -b$. Therefore, $b+\nh \in \{\pm c\}$, so up to relabelling $c$ and $-c$, it holds that $b+\nh = c$.
\end{proof}
\end{aid}

\begin{aid} \label{val7pf-NO-order4-NO-2a=2b-AID}
Proof that the graph $X_0 = \Cay(\ZZ_{4 \kappa},\{\pm \alpha, \pm(\alpha + 2\kappa), \pm(\alpha - \kappa)\})$ from \cref{val7pf-NO-order4-NO-2a=2b} is connected, nonbipartite and twin-free.
\begin{proof}
We recall that it has already been established that $\langle \alpha \rangle = \ZZ_n$ and that $\kappa$ and $\alpha$ are both odd. Connectedness and nonbipartiteness follow immediately. Assume for contradiction that $X_0$ is not twin-free. Then there exists a non-zero $h$ in $\ZZ_{4\kappa}$ such that:
\[h + \{\pm \alpha, \pm(\alpha + 2\kappa), \pm(\alpha - \kappa)\} = \{\pm \alpha, \pm(\alpha + 2\kappa), \pm(\alpha - \kappa)\}.\]
Since the connection set of $X_0$ is of cardinality $6$, it cannot contain equal number of even and odd elements, so it follows that $h+ S_o = S_o$ and $h+S_e=S_e$. In particular, since $\pm \alpha$ and $\pm (\alpha+2\kappa)$ are of the same parity (both are odd), it follows that $h + (\alpha-\kappa) = -(\alpha-\kappa)$ and $h - (\alpha-\kappa)=\alpha-\kappa$. We obtain that $2h=0$, so $h=2\kappa$. But then $2\kappa + (\alpha-\kappa)=-(\alpha-\kappa)$ implies that $2\alpha=0$ and $\alpha \in \{0,2\kappa\}$. This is a contradiction with $\alpha$ being odd.
\end{proof}
\end{aid}
\end{appendix}